\documentclass[12pt]{amsart}

\usepackage{geometry}
\usepackage[all]{xy}
\usepackage{epstopdf}
\usepackage{amssymb,amsmath,amsfonts,amsthm,graphicx,enumerate,amscd}
\usepackage{mathrsfs}
\usepackage{slashed}
\usepackage{xcolor}
\usepackage{hyperref}

\newcommand{\deff}[1]{\textbf{\emph{\sharp1}}}

\newcommand{\func}[3]{\sharp1 \colon \sharp2 \to \sharp3}
\newcommand{\bb}[1]{\mathbb{\sharp1}}
\newcommand{\lie}[1]{\mathfrak{\sharp1}}
\newcommand{\iprod}[2]{\langle \sharp1, \sharp2 \rangle}
\newcommand{\ddell}[1]{\frac{\partial}{\partial \sharp1}}

\theoremstyle{plain}
\newtheorem{theorem}{Theorem}[section]

\newtheorem{claim}[theorem]{Claim}
\newtheorem{corollary}[theorem]{Corollary}
\newtheorem{lemma}[theorem]{Lemma}
\newtheorem{proposition}[theorem]{Proposition}

\theoremstyle{definition}

\newtheorem{remark}[theorem]{Remark}

\newtheorem{definition}[theorem]{Definition}

\title[The space of commuting elements in $SU(2)$]{Cohomology and $K$-theory rings of the space of commuting elements in $SU(2)$}
\author{Chi-Kwong Fok}
\date{September 4, 2024}

\allowdisplaybreaks
\begin{document}
\maketitle
\begin{abstract}
	In this paper, we compute explicitly both the $K$-theory and integral cohomology rings of the space of commuting elements in $SU(2)$ via the $K$-theory of its desingularization. We also briefly discuss the different behavior of its cohomology with complex and $\mathbb{Z}_2$ coefficients in the context of representation stability and FI-modules.
\end{abstract}
\maketitle
\tableofcontents

\section{Introduction}
Let $G$ be a compact connected Lie group. The moduli space of flat principal $G$-bundles over a manifold $X$ is an important geometric object in mathematical physics due to its intimate relations with gauge theory and conformal field theory, which in turn shed light on its geometry and topology (cf. \cite{AB} and \cite{Be}). A coarser version of such a moduli space, i.e. the space of flat principal $G$-bundles modulo based gauge equivalence, can be identified with $\text{Hom}(\pi_1(X), G)$, the space of homomorphisms from the fundamental group of $X$ to $G$. When $X$ is an $n$-dimensional torus $(S^1)^n$, the fundamental group is $\mathbb{Z}^n$ and the corresponding moduli space is none other than the space of commuting $n$-tuples in $G$, which has garnered interest recently (\cite{AC, ACh, AG, AG2, AGG, B, BJS, Ba}). Note that $\text{Hom}(\mathbb{Z}^n, G)$ can be regarded as a real algebraic subvariety of $G^n$ cut out by the commuting condition. In general, the larger $n$ is, the more singular this algebraic subvariety becomes. 

In \cite{B}, the rational (equivariant) cohomology ring structure of $\text{Hom}(\mathbb{Z}^n, G)$ is presented in terms of the Lie algebra of a maximal torus of $G$ and the action of Weyl group, and is computed based on the observation that there is a map to $\text{Hom}(\mathbb{Z}^n, G)$ from one of its desingularizations (analogous to the Weyl covering map) whose fibers are rationally acyclic. The module structure of the rational equivariant $K$-theory of $\text{Hom}(\mathbb{Z}^n, G)$ for a large class of $G$ is given in \cite{AG} by specializing their more general result on the rational equivariant $K$-theory of spaces with maximal rank isotropy subgroups. As to the integral cohomology of $\text{Hom}(\mathbb{Z}^n, G)$, not as much progress has been made due to the issue of complicated singularities. In \cite{AC}, the suspension of $\text{Hom}(\mathbb{Z}^n, G)$ is decomposed homotopically into simpler pieces using a natural filtration, and based on this the cohomology groups of the special case $\text{Hom}(\mathbb{Z}^n, SU(2))$ for $n=2$ and 3 are explicitly given. Later, by further improving this suspension decomposition technique, a more explicit description of the homotopy type of $\Sigma\text{Hom}(\mathbb{Z}^n, SU(2))$ is given in \cite{BJS}, facilitating the computation of the cohomology group of $\text{Hom}(\mathbb{Z}^n, SU(2))$ for any $n$. 
In \cite{AG}, the module structure of integral equivariant $K$-theory of $\text{Hom}(\mathbb{Z}^2, SU(2))$ is computed by applying Segal spectral sequence to its equivariant CW-complex structure, while the algebra structure of integral equivariant $K$-theory and cohomology of the same space is found in \cite{Ba} using the similar approach of explicit analysis of the equivariant CW-structure.

In this paper, we give explicitly the ring structure of both the $K$-theory and integral cohomology of $\text{Hom}(\mathbb{Z}^n, SU(2))$. Our approach is arguably more elementary than the techniques previously employed. We first compute the $K$-theory of a desingularization of $\text{Hom}(\mathbb{Z}^n, SU(2))$, taking care to describe its vector bundles which represent the generators of the ring. An interesting feature about these vector bundles is that most of them are reduced line bundles whose tensor squares are isomorphic to the zero vector bundle. This turns out to enable us, despite the presence of torsions, to define the `integral Chern character map' from the $K$-theory to integral cohomology of the desingularization, and show that it is a ring isomorphism. Realizing $\text{Hom}(\mathbb{Z}^n, SU(2))$ by collapsing copies of $\mathbb{RP}^2$ from the desingularization and applying the relevant long exact sequence in $K$-theory lead to the desired $K$-theory ring structure, which is given in Theorem \ref{mainthm}. The cohomology ring structure (Corollary \ref{cohring}) can be got on the nose as the `integral Chern character map' on $\text{Hom}(\mathbb{Z}^n, SU(2))$ still makes sense and ring isomorphism persists in this case. We find that our results do agree with the cohomology group of $\text{Hom}(\mathbb{Z}^n, SU(2))$ deduced from the homotopy type of its suspension given in \cite{BJS} and the equivariant $K$-theory ring structure of $\text{Hom}(\mathbb{Z}^2, SU(2))$ computed in \cite{Ba}. 

Seeing that $\text{Hom}(\mathbb{Z}^n, SU(2))$ comes equipped with the natural $S_n$-action which permutes the $n$ commuting tuples and makes its cohomology group a $S_n$-representation, we also discuss the behavior of the cohomology group we obtain in the context of representation stability and $\text{FI}$-modules (\cite{CF}, \cite{CEF} and \cite{CEFN}). While the cohomology group with complex coefficients is known to be uniformly representation stable and hence is a finitely generated FI-module, we deduce that the cohomology group with $\mathbb{Z}_2$ coefficients is not a finitely generated FI-module (Section 5 and Corollary \ref{nonfgfimod}). 

\textbf{Acknowledgments} The author would like to thank Jos\'e Manuel Gom\'ez for generously sharing his notes and ideas on spaces of commuting elements in $SU(2)$, which are crucial to this paper. He is grateful to Lisa Jeffrey and Paul Selick for helpful discussions and pointing out \cite{Ba} to him when he visited the Fields Institute in the summer of 2018. He also would like to thank the referees for their meticulous comments on the drafts of this paper. This work is partially supported by the School of Mathematics and Physics Research Grant SRG2324-06 of the Xi'an Jiaotong-Liverpool University.
\section{The $K$-theory of the blowup}
\begin{definition}
	Let $G$ be $SU(2)$, $T$ its maximal torus of diagonal matrices and $\Gamma=\{1, \gamma\}$ the Weyl group, which is isomorphic to $\mathbb{Z}_2$. Let $Y_n=\text{Hom}(\mathbb{Z}^n, SU(2))$. Define the map $r: G/T\times_\Gamma T^n\to Y_n$ to be
	\[[(gT, t_1, \cdots, t_n)]\mapsto (gt_1g^{-1}, \cdots, gt_ng^{-1}).\]
Let $A$ be $\{(t_1, \cdots, t_n)\in T^n| t_i=\pm 1\text{ for }1\leq i\leq n\}$, i.e. $A$ is the subset of $T^n$ fixed by $\Gamma$. If $a\in A$, then denote the $i$-th coordinate of $a$ by $a_i$.
\end{definition}
Through the identification of $G/T$ with $S^2$ and $T$ with $S^1$, we can see that $\gamma$ acts on $S^2$ by the antipodal map and $S^1$ by reflection. The map $r$ when restricted to $G/T\times_\Gamma T^n\setminus G/T\times_\Gamma A$ is a diffeomorphism onto $Y_n\setminus \{(\pm I_2, \pm I_2, \cdots, \pm I_2)\}$, and collapses each of the $2^n$ real projective planes of the form $G/T\times_\Gamma\{a\}$, $a\in A$, to a point from the set $\{(g_1, g_2, \cdots, g_n)| g_i=I_2\text{ or }-I_2\}$ of tuples of singular elements of $G$. In this way, we can view $G/T\times_\Gamma T^n$ as the `blowup' of $Y_n$ as a real algebraic subvariety of $G^n$ at the singular points, and $Y_n$ can be got by collapsing $G/T\times_\Gamma \{a\}\subset G/T\times_\Gamma T^n$ to the point $a$. In this section and the next, as a first step towards understanding the $K$-theory and cohomology of $Y_n$, we shall investigate the $K$-theory ring structure of $G/T\times_\Gamma T^n$. 
We also define some vector bundles of interest over $G/T\times_\Gamma T^n$ along the way.

\begin{proposition}\label{intcohgp}
	The integral cohomology groups of $G/T\times_\Gamma T^n$ are given by
	\[H^i(G/T\times_\Gamma T^n, \mathbb{Z})=\begin{cases}\mathbb{Z}&\ \text{if }i=0,\\0&\ \text{if }i=1,\\ \mathbb{Z}^{\binom{n}{i-2}}&\ \text{if }i>1\text{ and odd, and}\\ \mathbb{Z}^{\binom{n}{i}}\oplus\mathbb{Z}_2^{\binom{n+1}{i-1}}&\ \text{if }i>0\text{ and even}.\end{cases}\]
	The odd $K$-theory group $K^{-1}(G/T\times_\Gamma T^n)$ is isomorphic to $\mathbb{Z}^{2^{n-1}}$, while the even $K$-theory group $K^0(G/T\times_\Gamma T^n)$ is isomorphic to $\mathbb{Z}^{2^{n-1}}\oplus M_n$ where $M_n$ is a finite abelian group of order $2^{2^n}$.
\end{proposition}
\begin{proof}
	Consider the fibration $T^n\hookrightarrow G/T\times_\Gamma T^n\to \mathbb{RP}^2$. The $E_2$-page of the Serre spectral sequence associated to this fibration is given by 
	\begin{align*}
		E_2^{p, q}&=H^p(\mathbb{RP}^2, \underline{H^q(T^n)})\\
				&=H^p(\mathbb{RP}^2, \underline{\bigwedge\nolimits^q_\mathbb{Z}(\mathbb{Z}^{\oplus n})})
	\end{align*}
	where $\underline{\bigwedge\nolimits^q_\mathbb{Z}(\mathbb{Z}^{\oplus n})}$ is the local coefficient system $S^2\times_\Gamma \bigwedge\nolimits^q_\mathbb{Z}(\mathbb{Z}^{\oplus n})$. Here each summand $\mathbb{Z}$ corresponds to $H^1(T, \mathbb{Z})$, and the $\Gamma$-action on $\mathbb{Z}$ is negation, which is induced by the reflection on $T$. When $q$ is even, the local coefficient system is trivial, so the $E_2$-page is
	\begin{align*}
		E_2^{p, q}&=H^p(\mathbb{RP}^2, \mathbb{Z})^{\oplus \binom{n}{q}}\\
				&=\begin{cases}\mathbb{Z}^{\oplus \binom{n}{q}}&\ \text{if }p=0, \\ 0&\ \text{if }p=1, \\ \mathbb{Z}_2^{\oplus \binom{n}{q}}&\ \text{if }p=2,\text{ and}\\ 0&\ \text{if }p>2.\end{cases}
	\end{align*}
	On the other hand, when $q$ is odd, the local coefficient system is twisted and equals $S^2\times_\Gamma \mathbb{Z}^{\oplus\binom{n}{q}}$, so the $E_2$-page is given by 
	\begin{align*}
		E_2^{p, q}&=H^p(\mathbb{RP}^2, \underline{\mathbb{Z}})^{\oplus\binom{n}{q}}\\
				&=\begin{cases}0&\ \text{if }p=0, \\ \mathbb{Z}_2^{\oplus\binom{n}{q}}&\ \text{if }p=1, \\ \mathbb{Z}^{\oplus \binom{n}{q}}&\ \text{if }p=2,\text{ and}\\ 0&\ \text{if }p>2.\end{cases}
	\end{align*}
	By \cite[Theorem 1.2]{AGPP}, the Serre spectral sequence $\widetilde{E}^{*, *}_r$ of the fibration $T^n\hookrightarrow S^\infty\times_\Gamma T^n\to \mathbb{RP}^\infty$ collapses on the $E_2$-page and there are no extension problems. Using the homomorphism of spectral sequences $\widetilde{E}^{*, *}_r\to E^{*, *}_r$ induced by the natural inclusions $S^2\hookrightarrow S^\infty$ and $\mathbb{RP}^2\hookrightarrow\mathbb{RP}^\infty$, we see that the original spectral sequence also collapses on the $E_2$-page and there are no extension problems. The first claim in the proposition then follows.
	
	As to the $K$-theory group the proof proceeds similarly. The Atiyah-Hirzebruch spectral sequence associated to the same fibration $T^n\hookrightarrow G/T\times_\Gamma T^n\to \mathbb{RP}^2$ resembles the aforementioned Serre spectral sequence and has the $E_2$-page $E_2^{p, q}=H^p(\mathbb{RP}^2, \underline{K^q(T^n)})=H^p(\mathbb{RP}^2, \underline{\bigwedge\nolimits_\mathbb{Z}^q(\mathbb{Z}^{\oplus n})})$. When $q=0$, 
	\[E_2^{p, 0}=\begin{cases}\mathbb{Z}^{\oplus 2^{n-1}},&\text{ if }p=0,\\ 0,&\text{ if }p=1, \\ \mathbb{Z}_2^{\oplus 2^{n-1}},&\text{ if }p=2, \\ 0,&\text{ if }p>2.\end{cases}\]
	When $q=-1$, 
	\[E_2^{p, -1}=\begin{cases}0,& \text{ if }p=0, \\ \mathbb{Z}_2^{\oplus 2^{n-1}},&\text{ if }p=1, \\ \mathbb{Z}^{\oplus 2^{n-1}},&\text{ if }p=2, \\ 0,&\text{ if }p>2.\end{cases}\]
	In fact the above spectral sequence also collapses on the $E_2$-page: the differential $d_2: E_2^{p, q}\to E_2^{p+2, q-1}$ vanishes for $(p, q)\neq (0, 0)$ because either $E_2^{p, q}$ or $E_2^{p+2, q-1}$ is 0. When $(p, q)=(0, 0)$, $d_2$ is the zero map as well for otherwise, the ranks of $E_\infty^{0, 0}$ and $E_\infty^{2, -1}$ would be strictly less than those of $E_2^{0, 0}$ and $E_2^{2, -1}$, and the dimension of $K^*(G/T\times_\Gamma T^n)\otimes\mathbb{Q}$ would also be strictly less than that $H^*(G/T\times_\Gamma T^n, \mathbb{Q})$, contradicting the isomorphism of the Chern character on finite CW-complexes. We have $E_\infty^{0, 0}=E_\infty^{2, -1}=\mathbb{Z}^{2^{n-1}}$ and $E_\infty^{2, 0}=E_\infty^{1, -1}=\mathbb{Z}_2^{2^{n-1}}$. The odd $K$-theory $K^{-1}(G/T\times_\Gamma T^n)$ is isomorphic to $E_\infty^{2, -1}\cong\mathbb{Z}^{2^{n-1}}$. The even $K$-theory $K^0(G/T\times_\Gamma T^n)$ can be obtained by the following group extensions
	\begin{align}
		0&\longrightarrow E_\infty^{2, 0}\longrightarrow M_n\longrightarrow E_\infty^{1, -1}\longrightarrow 0\label{extension2}\\
		0&\longrightarrow M_n\longrightarrow K^0(G/T\times_\Gamma T^n)\longrightarrow E_\infty^{0, 0}\longrightarrow 0. \label{extension}
	\end{align}
	So $M_n$ is a finite abelian group of order $2^{2^n}$, and the second extension splits. This completes the proof of the second claim of the proposition about the $K$-theory group of $G/T\times_\Gamma T^n$.
\end{proof}

\begin{definition}
	\begin{enumerate}
		\item By regarding $G/T\times_\Gamma T^n$ as a $T^n$-bundle over $(G/T)/\Gamma\cong\mathbb{RP}^2$, we let
	\[\pi: G/T\times_\Gamma T^n\to\mathbb{RP}^2\]
	be the projection map. By abuse of notation, let $\pi_{i}$ (resp. $\pi_{ij}$) be the projection map from $T^n$ or $G/T\times T^n$ onto the copy of $T$ in the $i$-th factor (resp. onto the product of circles from the $i$-th and $j$-th factors). We also use $\pi_i$ (resp. $\pi_{ij}$) to denote the projection maps $G/T\times_\Gamma T^n\to G/T\times_\Gamma T$ and $Y_n\to Y_1$ (resp. $G/T\times_\Gamma T^n\to G/T\times_\Gamma T^2$) which are similarly defined. Let 
	\[p: T^2\to S^2\] 
	be the map which collapses the longitudinal and latitudinal circles $\{(-1, e^{i\theta_2})|0\leq \theta_2\leq 2\pi\}\cup\{(e^{i\theta_1}, -1)|0\leq\theta_1\leq 2\pi\}$ to the south pole and sends the point $(1, 1)$ to the north pole. Let 
	\[p_{ij}: G/T\times T^n\to S^2\ (\text{or }p_{ij}: T^n\to S^2)\]
	be the composition $p\circ\pi_{ij}$. We also define the covering map 
	\[t: G/T\times T^n\to G/T\times_\Gamma T^n.\]
		\item Let $H$ be the hyperplane line bundle of $\mathbb{CP}^1\cong S^2$ equipped with the $\mathbb{Z}_2$-action which descends to the action $[z_0:z_1]\mapsto [z_0:-z_1]$ on the base, i.e. rotation by $\pi$ on $S^2$, acting on the fiber over the north pole $[1: 0]$ trivially, and on the fiber over the south pole $[0:1]$ by negation.	
		\item Let $\delta\in K^{-1}(T)$ be, through the identification $K^{-1}(T)=\widetilde{K}^0(S^2)$, the $K$-theory class $H-1$. Let $\delta_i=\pi_i^*\delta\in K^{-1}(T^n)$ (or by abuse of notation $\delta_i=\pi_i^*\delta\in K^{-1}(G/T\times T^n)$). 
	\end{enumerate}
\end{definition}

\begin{remark}
	Recall that the $K$-theory ring $K^*(S^2)$ is isomorphic to $\mathbb{Z}[H]/((H-1)^2)$, while $K^*(T^n)$ is isomorphic to $\bigwedge\nolimits^*_\mathbb{Z}(\delta_1, \delta_2, \cdots, \delta_n)$.
\end{remark}

\begin{proposition}\label{twodeltaprod}
	We have $p_{ij}^*(H-1)=\delta_i\delta_j\in \widetilde{K}^0(G/T\times T^n)$ for $i<j$. 
\end{proposition}
\begin{proof}
	It suffices to prove that $p^*(H-1)=\delta_1\delta_2\in \widetilde{K}^0(T^2)$, for $p_{ij}^*(H-1)=\pi_{ij}^*\circ p^*(H-1)=\pi_{ij}^*\delta_1\delta_2=\delta_i\delta_j$. Consider the commutative diagram
	\begin{eqnarray*}
		\xymatrix{\widetilde{K}^*(S^2)\ar[r]^{p^*}\ar[d]_{\text{ch}}& \widetilde{K}^*(T^2)\ar[d]^{\text{ch}}\\ \widetilde{H}^*(S^2, \mathbb{Z})\ar[r]^{p^*}&\widetilde{H}^*(T^2, \mathbb{Z})}
	\end{eqnarray*}
	Note that the two vertical maps, which are Chern character maps, are ring isomorphisms from (integral) $K$-theory to integral cohomology. Since $p^*: H^2(S^2, \mathbb{Z})\to H^2(T^2, \mathbb{Z})$ is an isomorphism, $p$ is orientation-preserving, and $\text{ch}(H-1)=c_1(H)$ is the (positive) generator of $\widetilde{H}^2(S^2, \mathbb{Z})$, $p^*c_1(H)$ is the (positive) generator of $\widetilde{H}^2(T^2, \mathbb{Z})$. It follows that the preimage $\text{ch}^{-1}(p^*c_1(H))\in \widetilde{K}(T^2)$ is the (positive) generator of $\widetilde{K}(T^2)$, which is $\delta_1\delta_2$. By the commutativity of the above square, we have the desired claim.
\end{proof}
\begin{definition}
	Noting that the map $p_{ij}$ is equivariant with respect to the $\Gamma$-action on $G/T\times T^n$ and the $\mathbb{Z}_2$-action on $S^2$ by rotation by $\pi$, we define, for $i<j$, $x_{ij}\in K^0(G/T\times_\Gamma T^n)$ to be the $K$-theory class of the reduced vector bundle
	\[(p_{ij}^*H)/\Gamma-G/T\times_\Gamma T^n\times\mathbb{C}\]
corresponding to the $K$-theory class $p_{ij}^*H-G/T\times T^n\times\mathbb{C}\in K_\Gamma^0(G/T\times T^n)$. Similarly, define $x_{ij, \Gamma}\in K_\Gamma^0(T^n)$ to be $p_{ij}^*H-T^n\times\mathbb{C}$. Let $x_{ji}=-x_{ij}$ and $x_{ii}=0$ (resp. $x_{ji, \Gamma}=-x_{ij, \Gamma}$ and $x_{ii, \Gamma}=0$).
\end{definition}
\begin{corollary}
	We have $t^*x_{ij}=\delta_i\delta_j$. 
\end{corollary}
\begin{proof}
	Recall from Definition \ref{varvectbund} that $x_{ij}=(p_{ij}^*H)/\Gamma-G/T\times_\Gamma T^n\times\mathbb{C}$. So 
	\[t^*x_{ij}=p_{ij}^*H-G/T\times T^n\times\mathbb{C}=p_{ij}^*(H-1)=\delta_i\delta_j,\] 
	with the last equality following from Proposition \ref{twodeltaprod}.
\end{proof}
\begin{proposition}\label{oddinject}
	The map $t^*: K^{-1}(G/T\times_\Gamma T^n)\to K^{-1}(G/T\times T^n)$ is injective. 
\end{proposition}
\begin{proof}
	Since $t$ is a covering map, $t^*: H^*(G/T\times_\Gamma T^n; \mathbb{Q})\to H^*(G/T\times T^n; \mathbb{Q})$ is an injection onto $H^*(G/T\times T^n; \mathbb{Q})^\Gamma$. By the naturality of the Chern character isomorphism and the freeness of $K^{-1}(G/T\times_\Gamma T^n)$ (cf. Proposition \ref{intcohgp}), $t^*: K^{-1}(G/T\times_\Gamma T^n)\to K^{-1}(G/T\times T^n)$ is also injective.
\end{proof}
Note that by the K\"unneth formula, $K^{-1}(G/T\times T)\cong K^0(G/T)\otimes K^{-1}(T)\cong\mathbb{Z}\cdot(H-1)\otimes\delta$. Recall that $\gamma$ is the antipodal map on $G/T\cong S^2$ and so $\gamma^*(H-1)=H^{-1}-1=1-H$. Since $\gamma$ is the reflection on $T$, it induces a reflection on the reduced suspension $\Sigma T$, which is homeomorphic to $S^2$. By identifying $\delta$ with $H-1$ using the suspension isomorphism $K^{-1}(T)\cong\widetilde{K}^0(\Sigma T)$, we have $\gamma^*\delta=-\delta$ as $\gamma^*(H-1)=H^{-1}-1=1-H$. Identify the generator $(H-1)\otimes\delta\in K^{-1}(G/T\times T)$ with $(H-1)\otimes(H-1)\in \widetilde{K}^0(G/T\times\Sigma T)$ through the suspension isomorphism. The $K$-theory class $(H-1)\otimes(H-1)$ then is represented by the virtual vector bundle $U:=\pi_{G/T}^*H\widehat{\otimes}\pi_{\Sigma T}^*H\oplus \underline{\mathbb{C}}-(\pi_{G/T}^*H\widehat{\otimes}\underline{\mathbb{C}}\oplus\underline{\mathbb{C}}\widehat{\otimes}\pi_{\Sigma T}^* H)$. The virtual vector bundle
\[U\oplus\gamma^*U\cong(\pi_{G/T}^*H\widehat{\otimes}\pi_{\Sigma T}^*H\oplus\pi_{G/T}^*H^{-1}\widehat{\otimes}\pi_{\Sigma T}^*H^{-1})-\underline{\mathbb{C}}^{\oplus 2}\]
is acted upon by $\gamma^*$ through swapping the summands of each term in the formal difference. By quotienting out the $\Gamma$-action, $U\oplus\gamma^*U$ descends to the virtual vector bundle
\[(\pi_{G/T}^*H\widehat{\otimes}\pi_{\Sigma T}^*H\oplus\pi_{G/T}^*H^{-1}\widehat{\otimes}\pi_{\Sigma T}^*H^{-1})/\Gamma-\underline{\mathbb{C}}^{\oplus 2}\]
on $G/T\times_\Gamma \Sigma T$.
\begin{definition}\label{defw}
Let $w$ be the $K$-theory class in $K^{-1}(G/T\times_\Gamma T)$ represented by the virtual vector bundle $(\pi_{G/T}^*H\widehat{\otimes}\pi_{\Sigma T}^*H\oplus\pi_{G/T}^*H^{-1}\widehat{\otimes}\pi_{\Sigma T}^*H^{-1})/\Gamma-\underline{\mathbb{C}}^{\oplus 2}$ on $G/T\times_\Gamma\Sigma T$. Define $w_i$ to be $\pi_i^*w\in K^{-1}(G/T\times_\Gamma T^n)$.
\end{definition}
\begin{remark}\label{doubling}
	It follows from Definition \ref{defw} that $t^*w_i=2(H-1)\otimes\delta_i\in K^{-1}(G/T\times T^n)$.
\end{remark}
\begin{proposition}\label{H3gen}
	The Chern character $\text{ch}(w)$ is a generator of $H^3(G/T\times_\Gamma T, \mathbb{Z})\cong\mathbb{Z}$. 
\end{proposition}
\begin{proof}
	Since $\text{ch}(H-1)$ is a generator of $H^2(G/T, \mathbb{Z})$ and $\text{ch}(\delta)$ a generator of $H^1(T, \mathbb{Z})$, $\text{ch}((H-1)\otimes\delta)$ is a generator of $H^3(G/T\times T, \mathbb{Z})$ which is the `volume form' of $G/T\times T$. On the other hand, $G/T\times_\Gamma T$ is an orientable closed 3-manifold and so $H^3(G/T\times_\Gamma T, \mathbb{Z})\cong\mathbb{Z}$ which is generated by the volume form of $G/T\times_\Gamma T$. Since $t: G/T\times T\to G/T\times_\Gamma T$ is a double covering map, $p^*: H^3(G/T\times_\Gamma T, \mathbb{Z})\to H^3(G/T\times T, \mathbb{Z})$ amounts to the multiplication by 2 map. Note that 
	\begin{align*}
		t^*\text{ch}(w)&=\text{ch}(t^*w)\\
					&=\text{ch}(2(H-1)\otimes\delta)\\
					&=2\text{ch}((H-1)\otimes\delta).
	\end{align*}
	It follows that $\text{ch}(w)$ is a generator of $H^3(G/T\times_\Gamma T, \mathbb{Z})$.
\end{proof}
\begin{remark}\label{H3gengen} By Propositions \ref{H3gen} and \ref{intcohgp}, we have that $H^3(G/T\times_\Gamma T^n, \mathbb{Z})$ is freely generated by $\text{ch}(w_i)=\pi_i^*\text{ch}(w)$, $1\leq i\leq n$.\end{remark}
\begin{proposition}\label{freepart}
	The $K$-theory classes $\{x_{ij}\}_{1\leq i, j\leq n}$ and $\{w_i\}_{i=1}^n$ satisfy the relations
	\begin{eqnarray*}
		x_{ij}+x_{ji}, x_{ii}, \{x_{ij}x_{k\ell}-\text{sgn}(\sigma) x_{\sigma(i)\sigma(j)}x_{\sigma(k)\sigma(l)}|\sigma\in S_4\},\\ \{w_iw_j\}_{1\leq i, j\leq n}, \{w_ix_{jk}-\text{sgn}(\sigma) w_{\sigma(i)}x_{\sigma(j)\sigma(k)}|\sigma\in S_3\}.
	\end{eqnarray*}
	The abelian group generated by 
	\begin{eqnarray*}1, \left\{x_{i_1i_2}\cdots x_{i_{2k-1}i_{2k}}\left|1\leq i_1<i_2<\cdots<i_{2k}\leq n, 1\leq k\leq \left\lfloor\frac{n}{2}\right\rfloor\right.\right\}, \\
	\{w_i| 1\leq i\leq n\}, \left\{w_ix_{j_1j_2}\cdots x_{j_{2k-1}j_{2k}}\left|1\leq i<j_1<j_2<\cdots<j_{2k}\leq n, 1\leq k\leq\left\lfloor\frac{n}{2}\right\rfloor\right.\right\}.
	\end{eqnarray*}
	is a free summand of $K^*(G/T\times_\Gamma T^n)$ which is isomorphic to $\mathbb{Z}^{2^n}$. 
	\end{proposition}
\begin{proof}
	First, we shall prove the relations among the generators. Note that 
	\begin{align*}
		t^*(w_ix_{jk})&=t^*(w_i)t^*(x_{jk})\\
				   &=(2(H-1)\otimes \delta_i)(\delta_j\delta_k)\\
				   &=\text{sgn}(\sigma)\cdot 2(H-1)\otimes\delta_{\sigma(i)}\delta_{\sigma(j)}\delta_{\sigma(k)}\\
				   &=\text{sgn}(\sigma) t^*(w_{\sigma(i)})t^*(x_{\sigma(j)\sigma(k)})\\
				   &=\text{sgn}(\sigma) t^*(w_{\sigma(i)}x_{\sigma(j)\sigma(k)})
	\end{align*}
	By Proposition \ref{oddinject}, we get the relation $w_ix_{jk}=\text{sgn}(\sigma) w_{\sigma(i)}x_{\sigma(j)\sigma(k)}$. 
	
	The collapsing of the Atiyah-Hirzebruch spectral sequence on the $E_2$-page as shown in the proof of Proposition \ref{intcohgp} shows that $w_i$ comes from $E_\infty^{2, -1}=E_2^{2, -1}=H^2(\mathbb{RP}^2, \underline{\bigwedge_\mathbb{Z}^\text{odd}(\mathbb{Z}^{\oplus n})})$. The latter is isomorphic to the subgroup $K^{-1}(G/T\times_\Gamma T^n)$ because it fits into the following group extensions
	\begin{align*}
		0&\longrightarrow E_\infty^{2, -1}\longrightarrow N_n\longrightarrow E_\infty^{1, 0}\longrightarrow 0\\
		0&\longrightarrow N_n\longrightarrow K^{-1}(G/T\times_\Gamma T^n)\longrightarrow E_\infty^{0, -1}\longrightarrow 0
	\end{align*}
	and $E_\infty^{1, 0}=E_\infty^{0, -1}=0$. It follows that $w_iw_j$ comes from $E_\infty^{4, -2}=E_2^{4, -2}=H^4(\mathbb{RP}^2, {\bigwedge_\mathbb{Z}^\text{even}(\mathbb{Z}^{\oplus n})})=0$, which corresponds to the zero subring of $K^*(G/T\times_{\Gamma}T^n)$. So $w_iw_j$ is 0.
	
	By Definition \ref{varvectbund} and Proposition \ref{intcohgp}, $x_{ij}$ can be obtained by quotienting the virtual vector bundle represented by $\delta_i\delta_j\in K^0(G/T\times T^n)$ by the $\Gamma$-action. Thus $x_{ij}x_{k\ell}$ can be obtained from that represented by $\delta_i\delta_j\delta_k\delta_l$, and we have $\delta_i\delta_j\delta_k\delta_l=\text{sgn}(\sigma)\delta_{\sigma(i)}\delta_{\sigma(j)}\delta_{\sigma(k)}\delta_{\sigma(l)}$. Thus $x_{ij}x_{k\ell}=\text{sgn}(\sigma) x_{\sigma(i)\sigma(j)}x_{\sigma(k)\sigma(l)}$. The relations $x_{ij}+x_{ji}=0$ and $x_{ii}=0$ come from Definition \ref{varvectbund}. This proves the first part of the proposition. 
	
	Next, we shall prove the abelian group structure asserted in the proposition. Observe that the group $E_2^{0, 0}$ in the Atiyah-Hirzebruch spectral sequence, which is $H^0(\mathbb{RP}^2, K^0(T^n))$, can also be thought of as the group $E_2^{0, 0}$ in the Segal spectral sequence for the equivariant $K$-theory $K_\Gamma^*(G/T\times T^n)$ (cf. \cite[Remark after Proposition 5.3]{Se}). This group contains $H^0(\mathbb{RP}^2, \mathbb{Z}\delta_i\delta_j)\cong\mathbb{Z}$ as a subgroup, and its generator should be represented by any element in $K_\Gamma^0(G/T\times T^n)$ such that it restricts to $x_{ij, \Gamma}\in K_\Gamma^0(T^n)$ (by Definition \ref{varvectbund} and Proposition \ref{intcohgp}, $x_{ij, \Gamma}$ is a $\Gamma$-equivariant lift of $\delta_i\delta_j\in K^0(T^n)$) and $1\in K_\Gamma^0(G/T)\cong K^0(\mathbb{RP}^2)$. One such element is $x_{ij}\in K^0(G/T\times_\Gamma T^n)\cong K_\Gamma^0(G/T\times T^n)$ (cf. the definition of $x_{ij}$ in Definition \ref{varvectbund}). We may further argue in a similar fashion that in general, the generator of $H^0(\mathbb{RP}^2, \mathbb{Z}\delta_{i_1}\delta_{i_2}\cdots\delta_{i_{2k-1}}\delta_{i_{2k}})$ may be represented by $x_{i_1i_2}\cdots x_{i_{2k-1}i_{2k}}$. Thus $E_2^{0, 0}$ corresponds to the abelian subgroup of $K^*(G/T\times_\Gamma T^n)$ generated by 
	\begin{eqnarray*}1, \left\{x_{i_1i_2}\cdots x_{i_{2k-1}i_{2k}}\left|1\leq i_1<i_2<\cdots<i_{2k}\leq n, 1\leq k\leq \left\lfloor\frac{n}{2}\right\rfloor\right.\right\}.
	\end{eqnarray*}
	Similarly, the group $E_2^{2, -1}$, which is $H^2(\mathbb{RP}^2, \underline{K^{-1}(T^n)})$, contains $H^2(\mathbb{RP}^2, \underline{\mathbb{Z}\delta_i})$ as a subgroup. By considering $E_2^{2, -1}$ in the Atiyah-Hirzebruch spectral sequence for $K^*(G/T\times T^n)$, which is $H^2(G/T, K^{-1}(T^n))$, and the map $t^*: K^*(G/T\times_\Gamma T^n)\to K^*(G/T\times T^n)$, we have that $t^*$ induces the map
	\[H^2(\mathbb{RP}^2, \underline{\mathbb{Z}\delta_i})\to H^2(G/T, \mathbb{Z}\delta_i).\]
	Note that $H^2(G/T, \mathbb{Z})$ and $H^2(\mathbb{RP}^2, \underline{\mathbb{Z}})$ are both isomorphic to $\mathbb{Z}$ and generated by the (twisted) volume form of $G/T$ and $\mathbb{RP}^2$ respectively, and $t: G/T\to\mathbb{RP}^2$ is a double covering map. So the above map is the multiplication by 2 map. The generator of $H^2(G/T, \mathbb{Z}\delta_i)$ corresponds to the $K$-theory class $(H-1)\otimes\delta_i\in K^{-1}(G/T\times T^n)$. It follows that the generator of $H^2(\mathbb{RP}^2, \underline{\mathbb{Z}\delta_i})$ corresponds to $w_i\in K^{-1}(G/T\times_\Gamma T^n)$, which is the preimage of $2(H-1)\otimes\delta_i$ under $t^*$ by Remark \ref{doubling}. In general, using a similar analysis, $H^2(\mathbb{RP}^2, \underline{K^{-1}(T^n)})$ corresponds to the abelian subgroup in $K^{-1}(G/T\times_\Gamma T^n)$ generated by 
	\begin{eqnarray*}
	\{w_i| 1\leq i\leq n\}, \left\{w_ix_{j_1j_2}\cdots x_{j_{2k-1}j_{2k}}\left|1\leq i<j_1<j_2<\cdots<j_{2k}\leq n, 1\leq k\leq\left\lfloor\frac{n}{2}\right\rfloor\right.\right\}.
	\end{eqnarray*}
	Thus 
	the abelian group stated in the proposition is isomorphic 
	to $E_2^{0, 0}\oplus E_2^{2, -1}\cong E_\infty^{0, 0}\oplus E_\infty^{2, -1}$, which corresponds to the free summand of $K^*(G/T\times_\Gamma T^n)$ by the collapsing of the spectal sequence and the splitting of the extension (\ref{extension}) in the proof of Proposition \ref{intcohgp}.
\end{proof}
\begin{remark}
	Note that the generator of the subgroup $H^0(\mathbb{RP}^2, \mathbb{Z}\delta_i\delta_j)$ of $E^{0, 0}_2$ can also be represented by $x_{ij}+a$, where $a$ is a torsion $K$-theory class coming from $M_n$ in the split short exact sequence (\ref{extension}) in the proof of Proposition \ref{intcohgp}. In fact, a choice of the representative for the generator corresponds to a splitting of (\ref{extension}). 
\end{remark}
The $K$-theory classes in Proposition \ref{freepart} are free generators of $K^*(G/T\times_\Gamma T^n)$. There are also other generators which are defined below and will be shown to give rise to the torsion part of $K^*(G/T\times_\Gamma T^n).$
\begin{definition}\label{varvectbund}
	Let $u_\Gamma\in K^0_\Gamma(T)$ be the image of the reduced nontrivial one dimensional complex representation of $\Gamma$ under the pullback map $K^*_\Gamma(\text{pt})\to K^*_\Gamma(T)$. Let $u$ be the $K$-theory class of $\mathbb{RP}^2$ representing the reduced vector bundle $G/T\times_\Gamma\mathbb{C}_1-\mathbb{RP}^2\times\mathbb{C}$, where $\mathbb{C}_1$ is the nontrivial complex 1-dimensional representation of $\Gamma$. By abuse of notation, we also let $u$ be the $K$-theory class of $G/T\times_\Gamma T^n$ representing the pullback of the said reduced vector bundle through the projection map $\pi$. 
	
	Let $v_\Gamma\in K_\Gamma^0(T)$ be the $K$-theory class of the reduced $\Gamma$-equivariant vector bundle on $T$, $E_1-E_0$, where $E_n$ is the trivial line bundle $T\times\mathbb{C}$ with the $\Gamma$-action being
	\[(e^{i\theta}, z)\mapsto (e^{-i\theta}, e^{in\theta}z).\]
	Let $v_i\in K^0(G/T\times_\Gamma T^n)$ be the $K$-theory class which corresponds to $\pi_i^*(E_1-E_0)\in K_\Gamma^0(G/T\times T^n)$ through the isomorphism $K^*(G/T\times_\Gamma T^n)\cong K^*_\Gamma(G/T\times T^n)$. Thus $v_i=\pi_i^*(E_1-E_0)/\Gamma$. Similarly, we let $v_{i, \Gamma}\in K_\Gamma^0(T^n)$ be $\pi_i^*v_\Gamma$.
		
	
\end{definition}
\begin{proposition}\label{equivKT}
	The equivariant $K$-theory ring $K^0_\Gamma(T)$ is isomorphic to 
	\[\mathbb{Z}[u_\Gamma, v_\Gamma]/(u_\Gamma(u_\Gamma+2), v_\Gamma(v_\Gamma+2), v_\Gamma(u_\Gamma+2)).\]
\end{proposition}
\begin{proof}
	Consider the long exact sequence of equivariant $K$-theory groups associated with the pair $(T, A)$
	\[\cdots\longrightarrow K_\Gamma^0(T, A)\longrightarrow K_\Gamma^0(T)\stackrel{i^*}{\longrightarrow} K_\Gamma^0(A)\longrightarrow \cdots.\]
	By definition, the group $K_\Gamma^0(T, A)$ is $\widetilde{K}_\Gamma^0(T/A)$. Here $T/A$ is homeomorphic to a figure eight, i.e. the wedge sum of two circles, which are swapped by $\gamma$. Since any (ordinary) complex vector bundles over a circle is trivial, the same is true of any complex vector bundle over a figure eight. It follows that any $\Gamma$-equivariant complex vector bundle over a figure eight is the pullback of a $\Gamma$-representation over the point of contact of the two circles. Thus $K_\Gamma^0(T/A)\cong R(\Gamma)$ and $\widetilde{K}_\Gamma^0(T/A)=0$. So the restriction map $i^*: K_\Gamma^0(T)\to K_\Gamma^0(A)=K_\Gamma^0(\{1\})\oplus K_\Gamma^0(\{-1\})\cong R(\Gamma)\oplus R(\Gamma)$ is injective. Let $c\in R(\Gamma)$ be the nontrivial 1-dimensional representation of $\Gamma$. Then $i^*(1)=(1, 1)$, $i^*(1+u_\Gamma)=(c, c)$ and $i^*(1+v_\Gamma)=(1, c)$ by Definition \ref{varvectbund}. The image of $i^*$ is the set of pairs of virtual $\Gamma$-representations of equal virtual dimensions $\{(k_1+k_2c, k_1+k_3+(k_2-k_3)c)\in R(\Gamma)\oplus R(\Gamma)|\ k_1, k_2, k_3\in\mathbb{Z}\}$: on the one hand, as $T$ is connected, the virtual dimensions of the $\Gamma$-representations over the two fixed points in the image of $i^*$ are the same. On the other hand, we have 
	\[i^*(k_1+k_2+k_2u_\Gamma-k_3v_\Gamma)=(k_1+k_2c, k_1+k_3+(k_2-k_3)c).\]
Thus $K_\Gamma^0(T)$ is generated by $u_\Gamma$ and $v_\Gamma$. The three relations among $u_\Gamma$ and $v_\Gamma$ can be obtained by passing them to $R(\Gamma)\oplus R(\Gamma)$ by $q$. For example, we observe that
\[i^*(v_\Gamma(u_\Gamma+2))=(0, c-1)\cdot(c+1, c+1)=(0, c^2-1)=(0, 0).\]
By the injectivity of $i^*$, we have $v_\Gamma(u_\Gamma+2)=0$.
\end{proof}
Before proving the theorem about the $K$-theory ring structure of the blowup $G/T\times_\Gamma T$, we first state the following useful lemma on classifying equivariant line bundles and the map sending a vector bundle to its determinant line bundle.
\begin{lemma}(\cite[Theorem A.1 and Lemma A.2]{HL})\label{HLlemma}
	Let $G$ be a compact Lie group acting on a compact manifold $M$. 
	\begin{enumerate}
		\item\label{equivPicard} The equivariant Picard group of isomorphism classes equivariant complex $G$-line bundles over $M$ is isomorphic to the equivariant cohomology group $H_G^2(M, \mathbb{Z})$ through the equivariant first Chern class map.
		\item\label{det} Let $\det: \text{Vect}_G(M)\to H_G^2(M, \mathbb{Z})$ be the map taking the isomorphism class of an equivariant complex $G$-vector bundle to the equivariant first Chern class of its determinant line bundle (or, equivalently, the isomorphism class of its determinant line bundle by virtue of the above). Then it descends to the map $\det: K_G^0(M)\to H_G^2(M, \mathbb{Z})$ satisfying the following properties: for equivariant (virtual) vector bundles $V$ and $W$, $\det(V\pm W)=\det(V)\otimes\det(W)^{\otimes\pm 1}$ and $\det(V\otimes W)=\det(V)^{\otimes \text{rank}W}\otimes \det(W)^{\otimes\text{rank}V}$. 
	\end{enumerate}
\end{lemma}
\begin{theorem}\label{premainthm}
	Additively, $K^0(G/T\times_\Gamma T^n)$ is isomorphic to $\mathbb{Z}^{2^{n-1}}\oplus\mathbb{Z}_2^{2^n}$ and generated by 
	\begin{eqnarray*}1, \left\{x_{i_1i_2}\cdots x_{i_{2k-1}i_{2k}}\left|1\leq i_1<i_2<\cdots<i_{2k}\leq n, 1\leq k\leq \left\lfloor\frac{n}{2}\right\rfloor\right.\right\}, \\ u, \{v_i|1\leq i\leq n\}, \left\{ux_{i_1i_2}\cdots x_{i_{2k-1}i_{2k}}\left| 1\leq i_1<i_2<\cdots< i_{2k}\leq n, 1\leq k\leq\left\lfloor\frac{n}{2}\right\rfloor\right.\right\}, \\ \left\{v_ix_{j_1j_2}\cdots x_{j_{2k-1}j_{2k}}\left|1\leq i<j_1<j_2<\cdots<j_{2k}, 1\leq k\leq \left\lfloor\frac{n-1}{2}\right\rfloor\right.\right\}, \end{eqnarray*}
	while $K^{-1}(G/T\times_\Gamma T^n)$ is isomorphic to $\mathbb{Z}^{2^{n-1}}$ and generated by 
	\begin{eqnarray*}\{w_i| 1\leq i\leq n\}, \left\{w_ix_{j_1j_2}\cdots x_{j_{2k-1}j_{2k}}\left|1\leq i<j_1<j_2<\cdots<j_{2k}\leq n, 1\leq k\leq\left\lfloor\frac{n}{2}\right\rfloor\right.\right\}.\end{eqnarray*}
	Moreover, we have the following list of relations in addition to those in Proposition \ref{freepart}.
	\begin{enumerate}
		\item $2u=2v_i=0$ for all $1\leq i\leq n$, 
		\item $u^2=uv_i=v_iv_j=0$ for all $1\leq i, j\leq n$, 
		\item $uw_i=v_iw_j=0$ for all $1\leq i, j\leq n$, 
		\item $x_{ij}^2=ux_{ij}$ for all $1\leq i<j\leq n$, and
		\item $ux_{ij}^2=v_ix_{jk}^2=ux_{ij}\cdot x_{jk}=v_ix_{jk}\cdot x_{k\ell}=0$.
	\end{enumerate}
	
\end{theorem}
\begin{remark}
	We have yet to figure out the product $x_{ij}x_{jk}$ which is missing from the above list of relations among the generators, but this is not necessary for the description of the $K$-theory and cohomology ring structure of $Y_n$ to be presented later on.
\end{remark}
\begin{proof}
	First, we would like to show that the generators of $H^2(\mathbb{RP}^2, \mathbb{Z}\cdot 1)$ and $H^1(\mathbb{RP}^2, \underline{\mathbb{Z}\delta_i})$, which are subgroups of $E_2^{2, 0}$- and $E_2^{1, -1}$-pages respectively of the Atiyah-Hirzebruch spectral sequence in the proof of Proposition \ref{intcohgp}, can be represented by $u$ and $v_i$ respectively. 
	\begin{claim}
		The generator of $H^2(\mathbb{RP}^2, \mathbb{Z}\cdot 1)$ as a subgroup of the $E_2^{2, 0}$-page of the spectral sequence corresponds to $u$.
	\end{claim}
		By \cite[Theorem]{At}, $K^*(\mathbb{RP}^2)=K^0(\mathbb{RP}^2)\cong\mathbb{Z}\oplus\mathbb{Z}_2$, where the nonzero 2-torsion is represented by the reduced vector bundle $G/T\times_\Gamma\mathbb{C}_1-\mathbb{RP}^2\times\mathbb{C}$. The Atiyah-Hirzebruch spectral sequence for $K^*(\mathbb{RP}^2)$ is known to collapse on the $E_2$-page, and the generator of $E_2^{2, 0}=H^2(\mathbb{RP}^2, K^0(\text{pt}))\cong\mathbb{Z}_2$ corresponds to the aforementioned vector bundle. By considering the pullback map $\pi^*: K^*(\mathbb{RP}^2)\to K^*(G/T\times_\Gamma T^n)$ and the induced map on the $E_2^{2, 0}$-pages 
		\[\pi^*: H^2(\mathbb{RP}^2, K^0(\text{pt}))\to H^2(\mathbb{RP}^2, \pi^*(K^0(\text{pt})))\cong H^2(\mathbb{RP}^2, \mathbb{Z}\cdot 1)\subseteq H^2(\mathbb{RP}^2, K^0(T^n)),\] 
		we have that the generator of $H^2(\mathbb{RP}^2, \mathbb{Z}\cdot 1)$ corresponds to $u$, which is $\pi^*(G/T\times_\Gamma\mathbb{C}_1-\mathbb{RP}^2\times\mathbb{C})$. 

	\begin{claim}\label{claim2}
		The generators of $H^1(\mathbb{RP}^2, \mathbb{Z}\cdot\delta_i)$ as a subgroup of the $E^{1, -1}_2$-page of the spectral sequence can be represented by $v_i$.
	\end{claim}
		By Lemma \ref{HLlemma} \ref{equivPicard}, the $\Gamma$-equivariant Picard group of isomorphism classes of complex $\Gamma$-line bundles on $T$ is isomorphic to $H_\Gamma^2(T, \mathbb{Z})$, which by the collapsing of the $E_2$-page of the fibration $T\hookrightarrow S^\infty\times_\Gamma T\to \mathbb{RP}^\infty$ without extension problems (\cite[Theorem 1.2]{AGPP}) fits into the split short exact sequence 
		\[0\longrightarrow E_2^{2, 0}\longrightarrow H^2_\Gamma(T, \mathbb{Z})\longrightarrow E_2^{1, 1}\longrightarrow 0\]
	where $E_2^{2, 0}=H_\Gamma^2(\text{pt}, H^0(T, \mathbb{Z}))\cong\mathbb{Z}_2$ and $E^{1, 1}_2=H_\Gamma^1(\text{pt}, \underline{H^1(T, \mathbb{Z})})\cong\mathbb{Z}_2$. So $H^2_\Gamma(T, \mathbb{Z})\cong\mathbb{Z}_2\oplus\mathbb{Z}_2$. By Proposition \ref{equivKT}, the four non-isomorphic complex $\Gamma$-line bundles over $T$ are given by the trivial line bundle 1, $1+u_\Gamma$, $1+v_\Gamma$, and $(1+u_\Gamma)\otimes (1+v_\Gamma)$ (see also \cite[Proof of item 3 of Lemma 5.1]{CKMS}). Noting that the generator of $H^2_\Gamma(\text{pt}, \mathbb{Z})\cong\mathbb{Z}_2$ corresponds to the nontrivial 1-dimensional representation of $\Gamma$ and the isomorphism $H^2_\Gamma(\text{pt}, \mathbb{Z})\to H^2_\Gamma(\text{pt}, H^0(T, \mathbb{Z}))$ which pulls 1-dimensional representations of $\Gamma$ back to complex $\Gamma$-line bundles over $T$, the generator of $H^2_\Gamma(\text{pt}, H^0(T, \mathbb{Z}))$ then corresponds to the line bundle $1+u_\Gamma$. The generator of $H^1_\Gamma(\text{pt}, \underline{H^1(T, \mathbb{Z})})$ can be represented by the line bundle $1+v_\Gamma$ (note that the generator can also be represented by the line bundle $(1+u_\Gamma)\otimes (1+v_\Gamma)$, and the two representatives correspond to two splittings of the above short exact sequence). Consider the following commutative diagram
	\begin{eqnarray}
		\xymatrix{K^0_\Gamma(T)\ar[r]\ar[d]_\det& K^0(G/T\times_\Gamma T)\ar[d]^\det\\ H_\Gamma^2(T, \mathbb{Z})\ar[r]^\beta& H^2(G/T\times_\Gamma T, \mathbb{Z}).}
	\end{eqnarray}
	where the horizontal maps are pullbacks and $\det$ takes a (virtual, equivariant) vector bundle to its determinant line bundle. 
	
	By analysing the map induced by $\beta$ on the $E_2$-pages of the spectral sequences for $H_\Gamma^2(T, \mathbb{Z})$ and $H^2(G/T\times_\Gamma T, \mathbb{Z})$, we see that $\beta$ is an isomorphism, which maps $1+u_\Gamma$ to $1+u$ and $1+v_\Gamma$ to $1+v_1$. Moreover, the generator of the $E^{1, 1}_2$-page $H^1(\mathbb{RP}^2, \underline{H^1(T, \mathbb{Z})})\cong\mathbb{Z}_2$ of $H^2(G/T\times_\Gamma T, \mathbb{Z})$ can be represented by $1+v_1$, the image under $\beta$ of $1+v_\Gamma$ which has just been shown to represent the generator of $H^1_\Gamma(\text{pt}, \underline{H^1(T, \mathbb{Z})})$. On the other hand, consider the $E_2$-page for the map induced by the right vertical map $\det$ from the torsion part $M_1=\text{Tors }K^0(G/T\times_\Gamma T)$ (see the exact sequence (\ref{extension2}) in the proof of Proposition \ref{intcohgp}) to $H^2(G/T\times_\Gamma T, \mathbb{Z})$ and the resulting commutative diagram of the map between two short exact sequences
	\begin{eqnarray}\label{torsioniso}
		\resizebox{13cm}{!}{\xymatrix{0\ar[r]& H^2(\mathbb{RP}^2, K^0(T))\ar[r]\ar[d]& M_1\ar[r]\ar[d]& H^1(\mathbb{RP}^2, \underline{K^{-1}(T)})\ar[r]\ar[d]& 0\\ 0\ar[r]& H^2(\mathbb{RP}^2, H^0(T, \mathbb{Z}))\ar[r]&H^2(G/T\times_\Gamma T, \mathbb{Z})\ar[r]& H^1(\mathbb{RP}^2, \underline{H^1(T, \mathbb{Z})})\ar[r]&0}}
	\end{eqnarray}
	Both the left and right vertical maps are isomorphisms (for example, the right vertical map is an isomorphism because $\det: K^{-1}(T)\cong \widetilde{K}^0(S^2)\to \widetilde{H}^2(S^2, \mathbb{Z})\cong H^1(T, \mathbb{Z})$ is an isomorphism). By the 5-lemma, the middle map $\det: M_1\to H^2(G/T\times_\Gamma T, \mathbb{Z})$ is an isomorphism. More precisely, $\det$ sends $u\in M_1$ to $1+u$ and $v_1$ to $1+v_1$: note that $M_1$ consists of reduced $K$-theory classes as it is torsion, and both $u$ and $v_1$ are reduced, and $\det(v_1)=\det((1+v_1)-1)=\det(1+v_1)\otimes \det(1)^{-1}=1+v_1$ by Lemma \ref{HLlemma} \ref{det} (similarly we also have $\det(u)=1+u$). 
	Consequently, the generator of $H^1(\mathbb{RP}^2, \underline{K^{-1}(T)})=H^1(\mathbb{RP}^2, \underline{\mathbb{Z}\delta_1})$ can be represented by $v_1$, the preimage under $\det$ of $1+v_1\in H^2(G/T\times_\Gamma T, \mathbb{Z})$, which has just been shown to represent the generator of $H^1(\mathbb{RP}^2, \underline{H^1(T, \mathbb{Z})})$. 
	More generally, for $K^*(G/T\times_\Gamma T^n)$, the generator of $H^1(\mathbb{RP}^2, \underline{\mathbb{Z}\delta_i})$ can be represented by $v_i$.  This finishes the proof of Claim \ref{claim2}

	We are now in a position to show the additive structure of $K^*(G/T\times_\Gamma T^n)$. With the additive structure of the abelian group stated in Proposition \ref{freepart} (a free summand of $K^*(G/T\times_\Gamma T^n)$ corresponding to $E^{0, 0}_2\oplus E^{2, -1}_2\cong\mathbb{Z}^{2^n}$) having been shown in the proof of Proposition \ref{freepart}), it remains to show the additive structure of $M_n:=\text{Tors }K^*(G/T\times_\Gamma T^n)$. Recall that the $E_2^{2, 0}$-page for $K^0(G/T\times_\Gamma T^n)$, which is a subgroup of $M_n$ (cf. short exact sequence (\ref{extension2})), is $H^2(\mathbb{RP}^2, K^0(T))=\bigoplus_{1<i_1<i_2<\cdots<i_{2k}<n}H^2(\mathbb{RP}^2, \mathbb{Z}\delta_{i_1}\delta_{i_2}\cdots \delta_{i_{2k}})$. Note that the map
	\[H^2(\mathbb{RP}^2, \mathbb{Z}\cdot 1)\otimes_\mathbb{Z}H^0(\mathbb{RP}^2, \mathbb{Z}\delta_{i_1}\delta_{i_2}\cdots\delta_{i_{2k}})\to H^2(\mathbb{RP}^2, \mathbb{Z}\delta_{i_1}\delta_{i_2}\cdots\delta_{i_{2k}})\]
	given by cup product and coefficient multiplication is an isomorphism. Thus $ux_{i_1i_2}\cdots x_{i_{2k-1}i_{2k}}$, the product of the two $K$-theory classes $u$ and $x_{i_1i_2}\cdots x_{i_{2k-1}i_{2k}}$, which correspond to the generators of $H^2(\mathbb{RP}^2, \mathbb{Z}\cdot 1)$ and $H^0(\mathbb{RP}^2, \mathbb{Z}\delta_{i_1}\delta_{i_2}\cdots\delta_{i_{2k}})$ respectively, corresponds to the generator of $H^2(\mathbb{RP}^2, \mathbb{Z}\delta_{i_1}\cdots\delta_{i_{2k}})$. Similarly, the $E^{1, -1}_2$-page is 
	\[H^1(\mathbb{RP}^2, \underline{K^{-1}(T^n)})=\bigoplus_{1<i<j_1<j_2<\cdots<j_{2k}<n}H^1(\mathbb{RP}^2, \underline{\mathbb{Z}\delta_{i}\delta_{j_1}\cdots\delta_{j_{2k}}}),\] 
	and the map
	\[H^1(\mathbb{RP}^2, \underline{\mathbb{Z}\delta_i})\otimes H^0(\mathbb{RP}^2, \mathbb{Z}\delta_{j_1}\cdots\delta_{j_{2k}})\to H^1(\mathbb{RP}^2, \underline{\mathbb{Z}\delta_i\delta_{j_1}\cdots\delta_{j_{2k}}})\]
	given by cup product and coefficient multiplication is an isomorphism. Thus $v_ix_{j_1j_2}\cdots x_{j_{2k-1}j_{2k}}$ represents the generator of $H^1(\mathbb{RP}^2, \underline{\mathbb{Z}\delta_i\delta_{j_1}\delta_{j_2}\cdots\delta_{j_{2k}}})\cong\mathbb{Z}_2$. Now note that by Lemma \ref{HLlemma} \ref{det}, $\det(2v_i)=\det(v_i)^{\otimes 2}=(1+v_i)^{\otimes 2}$, which is $\pi_i^*E_2/\Gamma$, but $E_2$ is trivial because of the following $\Gamma$-equivariant isomorphism of line bundles:
	\begin{align*}
		\underline{\mathbb{C}}&\to E_2\\
		(e^{i\theta}, z)&\mapsto (e^{i\theta}, e^{-i\theta}z).
	\end{align*}
	From the isomorphism of the middle vertical map of (\ref{torsioniso}), we have that $2v_i$, being the reduced $K$-theory class corresponding to the trivial line bundle under $\det$, must be 0. Similarly, $2u=0$ in $M_n$. It follows that $ux_{i_1i_2}\cdots x_{i_{2k-1}i_{2k}}$ and $v_ix_{j_1j_2}\cdots x_{j_{2k-1}j_{2k}}$ are 2-torsion, and that the short exact sequence (\ref{extension2}) splits because $M_n$ is generated by generators of $E^{2, 0}_2$ and $E^{1, -1}_2$, which are all 2-torsion. Thus $M_n$ is spanned, as a vector space over $\mathbb{Z}_2$, by 
	\begin{eqnarray*}u, \{v_i|1\leq i\leq n\}, \left\{ux_{i_1i_2}\cdots x_{i_{2k-1}i_{2k}}\left| 1\leq i_1<i_2<\cdots< i_{2k}\leq n, 1\leq k\leq\left\lfloor\frac{n}{2}\right\rfloor\right.\right\}, \\ \left\{v_ix_{j_1j_2}\cdots x_{j_{2k-1}x_{j_{2k}}}\left|1\leq i<j_1<j_2<\cdots<j_{2k}, 1\leq k\leq \left\lfloor\frac{n-1}{2}\right\rfloor\right.\right\}. \end{eqnarray*}
	This, together with the additive structure of the free part of $K^*(G/T\times_\Gamma T^n)$ explained in the proof of Proposition \ref{freepart}, gives the first part of the theorem on the additive structure of $K^*(G/T\times_\Gamma T^n)$.
	
	Next, we shall show the following relations involving torsion elements $u$ and $v_i$. 
	\begin{enumerate}
		\item $2u=0$ and $2v_i=0$: this has been shown in the preceding paragraph.
		\item\label{mixed} $u^2=0$, $uv_i=0$ and $v_i^2=0$: 
		Note that the map $\pi_1^*: K_\Gamma^0(T)\to K^0_\Gamma(G/T\times T)\cong K^0(G/T\times_\Gamma T)$ satisfies $\pi_1^*(u_\Gamma)=u$, $\pi_1^*(v_\Gamma)=v_1$. By Proposition \ref{equivKT}, $u(u+2)=v_1(v_1+2)=v_1(u+2)=0$. Since $2u=0$ and $2v_1=0$, the previous relations can be reduced to $u^2=v_1^2=uv_1=0$. More generally, $uv_i=v_i^2=0$. 
		 
		 $v_iv_j=0$ for $i\neq j$: consider the map 
		\[H^1(\mathbb{RP}^2, \underline{\mathbb{Z}\delta_i})\otimes H^1(\mathbb{RP}^2, \underline{\mathbb{Z}\delta_j})\to H^2(\mathbb{RP}^2, \mathbb{Z}\delta_i\delta_j)\]
		given by cup product and coefficient multiplication. The cup product of the generators of $H^1(\mathbb{RP}^2, \underline{\mathbb{Z}\delta_i})$ and $H^1(\mathbb{RP}^2, \underline{\mathbb{Z}\delta_j})$, which are represented by $v_i$ and $v_j$, \emph{corresponds} to the $K$-theory class $v_iv_j$. It suffices to show that the cup square of the generator $e\in H^1(\mathbb{RP}^2, \underline{\mathbb{Z}})$ is $0\in H^2(\mathbb{RP}^2, \mathbb{Z})$. In fact, $e$ is the Euler class of the non-orientable real line bundle $G/T\times_\Gamma\mathbb{R}_1$ (the use of the twisted coefficients $\underline{\mathbb{Z}}$ accounts for the non-orientability of the line bundle). Thus the Whitney product formula implies that $e^2$ is the Euler class of the direct sum of line bundles $G/T\times_\Gamma (\mathbb{R}_1\oplus\mathbb{R}_1)$, which is isomorphic to the trivial rank-2 real vector bundle over $\mathbb{RP}^2$. We have that $e^2=0$ as desired. 
		\item $uw_i=v_iw_j=0$ \emph{for any }$1\leq i, j\leq n$. Note $K^{-1}(G/T\times_\Gamma T^n)$ is a free abelian group by Propositions \ref{intcohgp}, while $uw_i$ and $v_iw_j$ are both 2-torsion elements. It follows that $uw_i$ and $v_iw_j$ are 0.
	\item\label{square} {$x_{ij}^2=ux_{ij}$}: Recall from Definition \ref{varvectbund} that the map
	\[p_{ij}^*: K_{\Gamma}^*(S^2)\to K_{\Gamma}^*(G/T\times T^n)\cong K^*(G/T\times_\Gamma T^n)\]
	sends $H-1$ to $x_{ij}$ (here we use 1 to denote the trivial line bundle $S^2\times\mathbb{C}$ over $S^2$ with $\Gamma$ acting on the factor $\mathbb{C}$ trivially). Consider the $S^1$-action on $S^2$ by rotation, the equivariant $K$-theory $K^*_{S^1}(S^2)$ and the map 
	\[r_{S^1}^{\Gamma}: K_{S^1}^*(S^2)\to K_\Gamma^*(S^2)\]
	induced by restriction of $S^1$-action to $\Gamma$-action. Let $t\in R(S^1)$ be the 1-dimensional representation of $S^1$ with weight 1 (and by abuse of notation, the trivial line bundle $S^2\times\mathbb{C}$ with $S^1$ acting on $\mathbb{C}$ with weight 1), and $H_{S^1}$ be the $S^1$-equivariant line bundle on $S^2$ which, non-equivariantly, is the canonical line bundle, and restricts to the trivial representation on the fiber over the north pole $N$ and the representation $t$ on the fiber over the south pole $S$. Then we have $r_{S^1}^\Gamma(H_{S^1})=H$ and $r_{S^1}^\Gamma(t)=1+u$. By the injectivity of the restriction map
	\[r_{NS}: K_{S^1}^*(S^2)\to K_{S^1}^*(N)\oplus K_{S^1}^*(S)\] 
	(due to equivariant formality of the $S^1$-action on $S^2$ and \cite[Theorem A4]{RK}), which takes $H_{S^1}$ to $(1, t)$, we have that, as $K$-theory classes in $K_{S^1}^*(S^2)$, 
	\[(H_{S^1}-1)(H_{S^1}-t)=0.\]
	Applying $r_{S^1}^\Gamma$ to the above relation, we have 
	\[(H-1)(H-(1+u))=0.\]
	Applying $p_{ij}$, we get $x_{ij}(x_{ij}-u)=0$ as desired.
	\item $ux_{ij}^2=v_ix_{jk}^2=ux_{ij}\cdot x_{jk}=v_ix_{jk}\cdot x_{k\ell}=0$. The first two are 0 by items \ref{mixed} and \ref{square} above. Note that $x_{ij}x_{jk}$ is a 2-torsion because $t^*(x_{ij}x_{jk})=\delta_i\delta_j\delta_j\delta_k=0$, and thus a linear combination of the 2-torsion generators, which all contains factors $u$ and $v_i$. By items \ref{mixed}, and \ref{square}, $ux_{ij}x_{jk}=v_ix_{jk}x_{k\ell}=0$
 	\end{enumerate}
\end{proof}

\section{The cohomology of the blowup}\label{coblowup}
For a finite CW-complex $X$, let $\mathcal{A}$ be a subring of $K^*(X)$ generated by $b_1, \cdots, b_n$, which are represented by (products of) reduced line bundles $L_1-1, \cdots, L_n-1$ over $X$ or its suspension $\Sigma X$, satisfying $L_i^{\otimes 2}\oplus 1\cong L_i^{\oplus 2}$ (implying that $b_i^2=0$). One can define the Chern character $\text{ch}$ mapping from $\mathcal{A}$ to the \emph{integral} cohomology $H^*(X, \mathbb{Z})$, and it still makes sense even for torsion $K$-theory classes. To be more precise, we note that by taking the total Chern class of both sides of the isomorphism $L_i^{\otimes 2}\oplus 1\cong L_i^{\oplus 2}$, we have $c_1(L_i)^2=0$, and regardless of whether $b_i$ is torsion, its Chern character can be defined the way the Chern character on rational $K$-theory is defined:
\begin{align*}
	\text{ch}(b_i)&=\text{ch}(L_i-1)\\
			   &=e^{c_1(L_i)}-1\\
			   &=\left(1+c_1(L_i)+\frac{c_1(L_i)^2}{2}+\cdots\right)-1\\
			   &=c_1(L_i).
\end{align*}
Here the higher order terms $\displaystyle \frac{c_1(L_i)^n}{n!}$ for $n\geq 2$ are defined to be 0 as $c_1(L_i)^2=0$. One may argue that these terms may actually represent nonzero $n!$-torsions and thus the Chern character is only defined up to some torsions. However, by defining more generally the higher order terms of the form $\displaystyle \frac{c_1(L_i)^kc_1(L_j)^{n-k}}{n!}$ with $k\geq 2$ or $n-k\geq 2$ to be 0, we can ensure that the Chern character thus defined is a ring homomorphism:
\begin{align*}
	\text{ch}(b_ib_j)&=\text{ch}(L_i\otimes L_j\oplus 1-(L_i\oplus L_j))\\
				&=e^{c_1(L_i\otimes L_j)}+1-(e^{c_1(L_i)}+e^{c_1(L_j)})\\
				&=\left(1+c_1(L_i)+c_1(L_j)+\frac{(c_1(L_i)+c_1(L_j))^2}{2}+\cdots\right)+1-(2+c_1(L_i)+c_1(L_j))\\
				&=\frac{c_1(L_i)^2+c_1(L_j)^2}{2}+c_1(L_i)c_1(L_j)\\
				&=c_1(L_i)c_1(L_j)\\
				&=\text{ch}(b_i)\text{ch}(b_j)
\end{align*}
\begin{proposition}\label{ringiso}
	The Chern character 
	\[\text{ch}: K^*(G/T\times_\Gamma T^n)\to H^*(G/T\times_\Gamma T^n, \mathbb{Z})\]
	where
	\[\text{ch}(u)=c_1(\pi^*(G/T\times_\Gamma\mathbb{C}_1)),\ \text{ch}(v_i)=c_1(\pi_i^*(E_1)),\ \text{ch}(x_{ij})=c_1(p_{ij}^*H/\Gamma)\]
	and $\text{ch}(w_i)$ is the $i$-th free generator of $H^3(G/T\times_\Gamma T^n, \mathbb{Z})$ (cf. Remark \ref{H3gengen}) is a well-defined ring isomorphism.
\end{proposition}
\begin{proof}
	The free abelian subgroup of $K^*(G/T\times_\Gamma T^n)$ generated by products of $w_i, 1\leq i\leq n$ and $x_{ij}, 1\leq i<j\leq n$ is a subring of $K^*(G/T\times_\Gamma T^n)$ and the Chern character $\text{ch}$ restricted to this subring is a ring homomorphism into $H^*(G/T\times_\Gamma T^n, \mathbb{Z})$ because $x_{ij}$ and $w_i$ are free generators of the rationalized $K$-theory ring $K^*(G/T\times_\Gamma T^n)\otimes\mathbb{Q}$ on which ch is well-known to be a ring homomorphism into $H^*(G/T\times_\Gamma T^n, \mathbb{Q})$. By Proposition \ref{freepart}, this subring is a free abelian group summand of $K^*(G/T\times_\Gamma T^n)$ which is of rank $2^n$. Identifying $K^*(G/T\times_\Gamma T^n)/\text{Tors}$ with this free abelian subgroup, we shall first show that $\text{ch}$ maps $K^*(G/T\times_\Gamma T^n)/\text{Tors}$ isomorphically onto a free abelian group summand of $H^*(G/T\times_\Gamma T^n, \mathbb{Z})$ of rank $2^n$, which is a maximal free abelian subgroup of $H^*(G/T\times_\Gamma T^n, \mathbb{Z})$. We have shown in the proof of Proposition \ref{freepart} that $\delta_i\delta_j\in H^0(\mathbb{RP}^2, \mathbb{Z}\delta_i\delta_j)\subseteq H^0(\mathbb{RP}^2, K^0(T^n))=E_2^{0, 0}$ is represented by $x_{ij}$, whereas the generator of $H^2(\mathbb{RP}^2, \underline{\mathbb{Z}\delta_i})$, which is a subgroup of $H^2(\mathbb{RP}^2, \underline{K^{-1}(T^n)})=E_2^{2, -1}$ corresponds to $w_i$. By the collapsing of both the spectral sequences of $K^*(G/T\times_\Gamma T^n)$ and $H^*(G/T\times_\Gamma T^n, \mathbb{Z})$ on the $E_2$-page and applying the Chern character on the $E_2$-page, the generator of $H^0(\mathbb{RP}^2, \mathbb{Z}\text{ch}(\delta_i\delta_j))$, which is a subgroup of the $E_2$-page of the spectral sequence for $H^*(G/T\times_\Gamma T^n, \mathbb{Z})$, i.e. 
	 \[E_2^{0, 2}=H^0(\mathbb{RP}^2, H^2(T^n, \mathbb{Z}))=\bigoplus_{1\leq i<j\leq n} H^0(\mathbb{RP}^2, \mathbb{Z}\text{ch}(\delta_i\delta_j)),\] is represented by $\text{ch}(x_{ij})$. Likewise, the generator of $H^2(\mathbb{RP}^2, \underline{\mathbb{Z}\text{ch}(\delta_i)})$, which is a subgroup of the $E_2$-page 
	 \[E_2^{2, 1}=H^2(\mathbb{RP}^2, \underline{H^1(T^n, \mathbb{Z})})=\bigoplus_{i=1}^nH^2(\mathbb{RP}^2, \underline{\mathbb{Z}\text{ch}(\delta_i)}),\] 
	corresponds to $\text{ch}(w_i)$. Note that 
	\[\bigoplus_{q\text{ even}}E_2^{0, q}\oplus \bigoplus_{q\text{ odd}}E_2^{0, q}=\bigoplus_{q\text{ even}}H^0(\mathbb{RP}^2, H^q(T^n, \mathbb{Z}))\oplus\bigoplus_{q\text{ odd}}H^2(\mathbb{RP}^2, \underline{H^q(T^n, \mathbb{Z})})\]
	is generated by products of generators of $H^0(\mathbb{RP}^2, \mathbb{Z}\text{ch}(\delta_i\delta_j))$ and $H^2(\mathbb{RP}^2, \underline{\mathbb{Z}\text{ch}(\delta_i)})$, and as shown in the proof of Proposition \ref{intcohgp}, 
	it corresponds to a free abelian group summand of $H^*(G/T\times_\Gamma T^n, \mathbb{Z})$ of rank $2^n$. It follows that the free abelian subgroup of $H^*(G/T\times_\Gamma T^n, \mathbb{Z})$ generated by products of $\text{ch}(x_{ij})$ and $\text{ch}(w_i)$ (i.e. the image under the Chern character of the free abelian subgroup of $K^*(G/T\times_\Gamma T^n)$ generated by products of $w_i$ and $x_{ij}$) is a summand of rank $2^n$.
	
	Next, we shall show that the Chern character maps the subring $\text{Tors }K^*(G/T\times_\Gamma T^n)$ isomorphically onto $\text{Tors }H^*(G/T\times_\Gamma T^n, \mathbb{Z})$. Note that it makes sense to define ch on $\text{Tors }K^*(G/T\times_\Gamma T^n)$ and ch is a ring homomorphism on $\text{Tors }K^*(G/T\times_\Gamma T^n)$ because it is generated by products of $u$, $v_i$ and $x_{ij}$, all of which are represented by reduced line bundles (see Definition \ref{varvectbund}) satisfying the condition laid out in the discussion preceding this proposition. Besides the image of $\text{Tors }K^*(G/T\times_\Gamma T^n)$ under ch lies in $\text{Tors }H^*(G/T\times_\Gamma T^n, \mathbb{Z})$. Noting that both $\text{Tors }K^*(G/T\times_\Gamma T^n)$ and $\text{Tors }H^*(G/T\times_\Gamma T^n, \mathbb{Z})$ have the same order by Proposition \ref{intcohgp}, and bearing in mind that ch is an injective map, we have that $\text{ch}: \text{Tors }K^*(G/T\times_\Gamma T^n)\to \text{Tors }H^*(G/T\times_\Gamma T^n, \mathbb{Z})$ is a ring isomorphism. 
	
	Now we have shown that $\text{ch}: K^*(G/T\times_\Gamma T^n)\to H^*(G/T\times_\Gamma T^n, \mathbb{Z})$ is bijective, and that ch is a ring homomorphism when restricted to $K^*(G/T\times_\Gamma T^n)/\text{Tors}$ and $\text{Tors }K^*(G/T\times_\Gamma T^n)$. It remains to show that ch is a ring homomorphism on the full $K$-theory ring $K^*(G/T\times_\Gamma T^n)$. In particular, we need to check the following equations involving $w_i$ which is not represented by a reduced line bundle as well as $u$ and $v_i$ which are torsion: 
	\begin{align*}
		\text{ch}(uw_i)&=\text{ch}(u)\text{ch}(w_i),\\
		\text{ch}(v_iw_j)&=\text{ch}(v_i)\text{ch}(w_j),\\
		\text{ch}(w_iw_j)&=\text{ch}(w_i)\text{ch}(w_j).
	\end{align*}
	The LHS of the above equations are 0 because $uw_i=v_iw_j=w_iw_j=0$ by Proposition \ref{freepart} and Theorem \ref{premainthm} while the RHS can be written respectively as $\pi_i^*(\text{ch}(u)\text{ch}(w)),\ \pi_{ij}^*(\text{ch}(v_1)\text{ch}(w_2))$ (or $ \pi_i^*(\text{ch}(v)\text{ch}(w))$ if $i=j)$, $\pi_{ij}^*(\text{ch}(w_1)\text{ch}(w_2))$ (or $\pi_i^*(\text{ch}(w))^2$ if $i=j$). Note that, for example, $\text{ch}(v_1)\text{ch}(w_2)$ lives in $H^5(G/T\times_\Gamma T^2, \mathbb{Z})$, which is 0 because $\text{dim }G/T\times_\Gamma T^2=4<5$. Using this argument of comparing cohomological degree and dimension one can also show that $\text{ch}(u)\text{ch}(w)$, $\text{ch}(v)\text{ch}(w)$, $\text{ch}(w_1)\text{ch}(w_2)$ and $(\text{ch}(w))^2$ are 0 and hence the RHS of the above displayed equations are all 0. Now we have shown that $\text{ch}: K^*(G/T\times_\Gamma T^n)\to H^*(G/T\times_\Gamma T^n, \mathbb{Z})$ is a ring isomorphism as desired.  
\end{proof}

\section{The $K$-theory and cohomology of $\text{Hom}(\mathbb{Z}^n, SU(2))$}

\begin{definition}
	For $a\in A$, let $i_a: G/T\times_\Gamma\{a\}\hookrightarrow G/T\times_\Gamma T^n$ and $i: G/T\times_\Gamma A\hookrightarrow G/T\times_\Gamma T^n$ be natural inclusions. Let $x\in K^*(G/T\times_\Gamma T^n)$ be any $K$-theory class satisfying $i^*x=0$. So $x$ is a reduced $K$-theory class and can be represented by the virtual vector bundle $F-\underline{\mathbb{C}}^m$ for some rank $m$ vector bundle $F$ on $G/T\times_\Gamma T^n$ (respectively $\Sigma(G/T\times_\Gamma T^n)$) which is stably trivial when restricted on $G/T\times_\Gamma A$ (respectively its suspension), i.e. $F\oplus \underline{\mathbb{C}}^r$ is a trivial vector bundle for some $r$. Then we define the map
	\[r_*: \text{ker}(i^*)\to K^*(Y_n, A)\cong\widetilde{K}^*(Y_n/A)\]
	and $r_*(x)\in \widetilde{K}^*(Y_n/A)$ to be the $K$-theory class represented by the reduced vector bundle on $Y_n/A$ or $\Sigma(Y_n/A)$ obtained from $F-\underline{\mathbb{C}}^m=(F\oplus \underline{\mathbb{C}}^r)-\underline{\mathbb{C}}^{m+r}$ by collapsing the restricted trivial vector bundle $i^*(F\oplus\underline{\mathbb{C}}^r)$ over $G/T\times_\Gamma A$ to one of its fibers over the base point of $Y_n/A$ (cf. proof of \cite[Lemma 2.4.2]{At}). By abuse of notation we also denote by $r_*$ the maps $\text{ker}(i^*)\to K^*(Y_n)$ which are similarly defined by collapsing, for each point $a\in A$, the trivial vector bundle $i_a^*(F\oplus\underline{\mathbb{C}}^r)$ over $G/T\times_\Gamma \{a\}$ to one of its fibers over the point $a\in Y_n$.
\end{definition}

\begin{proposition}\label{restriction}
	We have $i_a^*v_i=\begin{cases}u&\ \ \text{if }a_i=-1\\ 0&\ \ \text{otherwise}\end{cases}$ and $\displaystyle i^*_ax_{ij}=\begin{cases}u&\ \ \text{if }a_i=-1\text{ or }a_j=-1\\ 0&\ \ \text{otherwise}\end{cases}$.
\end{proposition}
\begin{proof}
	Recall that $v_i$ is represented by the reduced line bundle which is the quotient of $G/T\times T^n\times\mathbb{C}_1-G/T\times T^n\times\mathbb{C}_0$ by the $\Gamma$-action given by 
	\begin{align*}
		G/T\times T^n\times\mathbb{C}_n&\to G/T\times T^n\times\mathbb{C}_n\\
		(x, e^{i\theta_1}, \cdots, e^{i\theta_n}, z)&\mapsto (\gamma(x), e^{-i\theta_1}, \cdots, e^{-i\theta_n}, e^{in\theta_i}z)
	\end{align*}
	Thus $i_a^*v_i$ is represented by the quotient of $G/T\times\{a\}\times\mathbb{C}_1-G/T\times\{a\}\times\mathbb{C}_0$ by the restricted $\Gamma$-action. If $a_i=-1$, then $G/T\times\{a\}\times\mathbb{C}_1/\Gamma$ is a nontrivial line bundle over $\mathbb{RP}^2$, and $(G/T\times\{a\}\times\mathbb{C}_1-G/T\times\{a\}\times\mathbb{C}_0)/\Gamma$ represents $u\in K^0(\mathbb{RP}^2)$. If $a_i=1$, then $G/T\times\{a\}\times\mathbb{C}_1/\Gamma$ is a trivial line bundle over $\mathbb{RP}^2$, and $(G/T\times\{a\}\times\mathbb{C}_1-G/T\times\{a\}\times \mathbb{C}_0)/\Gamma$ represents $0\in K^0(\mathbb{RP}^2)$.
	
	As to the restriction of $x_{ij}$ at $a$, we first recall that by Definition \ref{varvectbund}, $x_{ij}=[p_{ij}^*H/\Gamma-G/T\times_\Gamma T^n\times\mathbb{C}]$, where $H$ is the canonical line bundle over $S^2$ and equipped with the $\mathbb{Z}_2$-action which descends to rotation by $\pi$ on $S^2$, acts on the fiber over the north pole trivially and that over the south pole by negation. We also recall that the map $p_{ij}: G/T\times T^n\to S^2$ is the composition of projection onto the torus which is the product of the $i$-th and $j$-th circles, and the map which collapses the union of the longitudinal and latitudinal circles $\{(e^{i\theta}, -1)\in T_i\times T_j|0\leq \theta\leq 2\pi\}\cup\{(-1, e^{i\theta})\in T_i\times T_j|0\leq \theta\leq 2\pi\}$, which includes $(-1, 1)$, $(1, -1)$ and $(-1, -1)$, to the south pole of $S^2$ and maps $(1, 1)$ to the north pole. Thus if $a_i=-1$ or $a_j=-1$, $\Gamma$ acts on $i_a^*p_{ij}^*H$ by 
	\begin{align*}
		G/T\times\{a\}\times\mathbb{C}&\to G/T\times\{a\}\times\mathbb{C}\\
		(x, a, z)&\mapsto (\gamma(x), a, -z)
	\end{align*} 
	whereas if $a_i=a_j=1$, $\Gamma$ acts on $i_a^*p_{ij}^*H$ by 
	\begin{align*}
		G/T\times\{a\}\times\mathbb{C}&\to G/T\times\{a\}\times\mathbb{C}\\
		(x, a, z)&\mapsto (\gamma(x), a, z).
	\end{align*}
	This establishes the claim about $i_a^*x_{ij}.$
\end{proof}

\begin{proposition}\label{restriction2}
	The images of $\{x_{ij}\}_{1\leq i<j\leq n}$, $u$ and $\{v_i\}_{1\leq i\leq n}$ under the restriction map
	\[\widetilde{i}^*: \widetilde{K}^0(G/T\times_\Gamma T^n)\to\bigoplus_{a\in A}\widetilde{K}^0(G/T\times_\Gamma \{a\})\cong \bigoplus_{a\in A}\widetilde{K}^0(\mathbb{RP}^2)\]
	are linearly independent in $\displaystyle\bigoplus_{a\in A}\widetilde{K}^0(\mathbb{RP}^2)$ viewed as a $\mathbb{Z}_2$-vector space of $2^n$ dimensions.
\end{proposition}
\begin{proof}
	It is sufficient to show that for any subsets $I\subseteq \{1, 2, \cdots, n\}$ and $J\subseteq\{(i, j)\in\mathbb{Z}^2|1\leq i<j\leq n\}$, 
	\begin{align*}
		\sum_{i\in I}i^*v_i+\sum_{(i, j)\in J}i^*x_{ij}&\neq 0,\ \text{and}\\
		\sum_{i\in I}i^*v_i+\sum_{(i, j)\in J}i^*x_{ij}&\neq i^*u.
	\end{align*}
	By Proposition \ref{restriction}, this is equivalent to, for any given subsets $I$ and $J$, the existence of $a\in A$ such that
	\[|\{i\in I|a_i=-1\}|+|\{(i, j)\in J|a_i=-1\text{ or }a_j=-1\}|\text{ is odd}\]
	and the existence of $b\in A$ such that
	\[|\{i\in I|b_i=-1\}|+|\{(i, j)\in J|b_i=-1\text{ or }b_j=-1\}|\text{ is even}.\]
	An obvious choice of $b$ is the one satisfying $b_i=1$ for all $1\leq i\leq n$. Now we shall show that for any given subsets $I$ and $J$, we can choose $a$ which satisfies the above condition.
	Define 
	\[m_k=1+|\{(i, j)\in J| k=i\text{ or }k=j\}|.\]
	we choose $a$ by the following algorithm.
	\begin{enumerate}
		\item Suppose that in $I$, there exists $k$ such that $m_{k}$ is odd. Then we can choose $a\in A$ where $a_{k}=-1$ and $a_i=1$ for $i\neq k$.
		\item Suppose that $m_k$ is even for all $k\in I$. To choose $a$ we do the following. 
		\begin{enumerate}
			\item If there exist $k_1$ and $k_2\in I$ such that $(k_1, k_2)\in J$, then choose $a$ where $a_{k_1}=a_{k_2}=-1$ and $a_i=1$ for $i\neq k_1$ and $i\neq k_2$. 
			\item Otherwise, choose any $\ell$ such that $(k_1, \ell)$ or $(\ell, k_1)$ is in $J$. Then $\ell\notin I$.
			\begin{enumerate}
				\item If $m_\ell-1$ is odd, then choose $a$ where $a_\ell=-1$ and $a_i=1$ for $i\neq \ell$.
				\item If $m_\ell-1$ is even, then choose $a$ where $a_\ell=a_{k_1}=-1$ and $a_i=1$ for $i\neq \ell$ and $i\neq k_1$.
			\end{enumerate}
		\end{enumerate}
	\end{enumerate}
\end{proof}

\begin{theorem}\label{mainthm}
	The $K$-theory ring $K^*(Y_n)$ is generated by the following: 
	\begin{itemize}
		\item $a_{ij}:=r_*(2x_{ij})$, 
		\item $b_{ij}:=r_*(ux_{ij})$, $b'_{ijk\ell}:=r_*(ux_{ij}x_{k\ell})$, 
		\item $c_{ijk}:=r_*(v_ix_{jk})$, $c'_{ijk\ell m}:=r_*(v_ix_{jk}x_{lm})$, 
		\item $d_{ijk\ell}:=r_*(x_{ij}x_{k\ell})$, $d'_{ijk\ell mq}:=r_*(x_{ij}x_{k\ell}x_{mq})$
		\item $e_i:=r_*(w_i)$, $e'_{ijk}:=r_*(w_ix_{jk})$, and 
		\item $f_i\in K^{-1}(Y_n)$, $1\leq i\leq 2^n-1-n-\binom{n}{2}$
	\end{itemize}
	for $1\leq i<j<k<\ell<m<q\leq n$. 
	Additively, $K^0(Y_n)$ is isomorphic to $\mathbb{Z}^{2^{n-1}}\oplus\mathbb{Z}_2^{2^n-1-n}$ and generated by 
		\begin{itemize}
			\item $1$, $\{r_*(2x_{i_1i_2})\}_{1\leq i_1<i_2\leq n}, \{r_*(x_{i_1i_2}x_{i_3i_4}\cdots x_{i_{2k-1}i_{2k}})|1\leq i_1<i_2<\cdots<i_{2k}\leq n,\ k> 1\}$ ($2^{n-1}$ generators for the free subgroup),
			\item $\{r_*(ux_{i_1i_2}x_{i_3i_4}\cdots x_{i_{2k-1}i_{2k}})|1\leq i_1<i_2<\cdots<i_{2k}\leq n,\ k\geq 1\}$ ($2^{n-1}-1$ 2-torsion generators), and
			\item $\{r_*(v_{i_1}x_{i_2i_3}x_{i_4i_5}\cdots x_{i_{2k}i_{2k+1}})|1\leq i_1<\cdots<i_{2k+1}\leq n,\ k\geq 1\}$ ($2^{n-1}-n$ 2-torsion generators).
	\end{itemize}
	while $K^{-1}(Y_n)$ is isomorphic to $\mathbb{Z}^{2^{n-1}}\oplus \mathbb{Z}_2^{2^n-1-n-\binom{n}{2}}$ and generated by 
	\begin{itemize}
		\item the 2-torsion generators $f_1, \cdots, f_{2^n-1-n-\binom{n}{2}}$, 
		\item $\{r_*(w_i)| 1\leq i\leq n\}$, $\{r_*(w_{i_1}x_{i_2i_3}\cdots x_{i_{2k}i_{2k+1}})|1\leq i_1<i_2<\cdots<i_{2k+1}\leq n, k\geq 1\}$ ($2^{n-1}$ free generators).
	\end{itemize}
	The relations among the generators $a_{ij}$, $b_{ij}$, $b'_{ijk\ell}$, $c_{ijk}$, $c'_{ijk\ell m}$, $d_{ijk\ell}$, $d'_{ijk\ell mq}$, $e_i$ and $e'_{ijk}$ can be read off from those asserted in Theorem \ref{premainthm} seeing that $r_*$ is a ring homomorphism, e.g. $a_{ij}a_{k\ell}=4d_{ijk\ell}$, $a_{ij}^2=0$, $b_{ij}c_{k\ell m}=0$. Any product of generators involving $f_i$ is 0, and $2f_i=0$.
\end{theorem}
\begin{remark}\label{reducedrelative}
	Observe that $K^*\left(G/T\times_\Gamma T^n, G/T\times_\Gamma A\right)\cong K^*(Y_n, A)$. Using the long exact sequence for the pair $(Y_n, A)$, we have 
	\begin{eqnarray}\label{eqn5}K^*(Y_n, A)\cong \widetilde{K}^*(Y_n)\oplus \text{im}(d: \widetilde{K}^0(A)\to K^{-1}(Y_n, A)).\end{eqnarray}
	In view of the above isomorphism, we will sometimes regard elements of $\widetilde{K}^*(Y_n)$ as those of the relative $K$-theory $K^*(Y_n, A)$ and $K^*(G/T\times_\Gamma T^n, G/T\times_\Gamma A)$. The same goes for cohomology, as we also have
	\[H^*(Y_n, A, \mathbb{Z})\cong\widetilde{H}^*(Y_n)\oplus\text{im}(d: \widetilde{H}^0(A, \mathbb{Z})\to H^1(Y_n, A, \mathbb{Z})).\]
\end{remark}
\begin{proof}
	Note that the quotients $Y_n/A$ and $G/T\times_\Gamma T^n/(G/T\times_\Gamma A)$ are homeomorphic. In view of this, we will first compute $K^*(G/T\times_\Gamma T^n, G/T\times_\Gamma A)\cong K^*(Y_n, A)$, and then $K^*(Y_n)$ using the long exact sequence for the pair $(Y_n, A)$. Consider the following $K$-theoretic long exact sequence associated with the pair $\left(G/T\times_\Gamma T^n, G/T\times_\Gamma A\right)$ (rolled up in a loop due to the Bott periodicity).
	\begin{eqnarray*}
		\resizebox{15cm}{!}{\xymatrix@+4pc{K^0(G/T\times_\Gamma T^n, G/T\times_\Gamma A)\ar[r]^{j^*}& \widetilde{K}^0(G/T\times_\Gamma T^n)\ar[r]^{i^*}&\widetilde{K}^0(\coprod_{a\in A}\mathbb{RP}^2)\ar[d]^{d}\\
		\bigoplus_{a\in A}K^{-1}(\mathbb{RP}^2)\ar[u]^d& K^{-1}(G/T\times_\Gamma T^n)\ar[l]^{i^*}& K^{-1}(G/T\times_\Gamma T^n, G/T\times_\Gamma A)\ar[l]^{j^*}}}
	\end{eqnarray*}
	By exactness, since $\bigoplus_{a\in A}K^{-1}(\mathbb{RP}^2)=0$, $K^0(G/T\times_\Gamma T^n,  G/T\times_\Gamma A)$ is $\text{ker}(i^*: \widetilde{K}^0(G/T\times_\Gamma T^n)\to\widetilde{K}^0\left(\coprod_{a\in A}\mathbb{RP}^2\right)$, which by Theorem \ref{premainthm} and Proposition \ref{restriction} is, as a ring, generated by $r_*(2x_{ij})$, $r_*(ux_{ij})$, $r_*(ux_{ij}x_{k\ell})$, $r_*(v_{i}x_{jk})$, $r_*(v_{i}x_{jk}x_{\ell m})$, $r_*(x_{ij}x_{k\ell})$, and $r_*(x_{ij}x_{k\ell}x_{mq})$ for $1\leq i<j<k<\ell<m<q\leq n$, and additively isomorphic to $\mathbb{Z}^{2^{n-1}-1}\oplus \mathbb{Z}_2^{2^n-1-n}$ with the following generators
		\begin{itemize}
			\item 1, $\{r_*(2x_{i_1i_2})\}_{1\leq i_1<i_2\leq n}, \{r_*(x_{i_1i_2}x_{i_3i_4}\cdots x_{i_{2k-1}i_{2k}})|1\leq i_1<i_2<\cdots<i_{2k}\leq n,\ k>1\}$ ($2^{n-1}$ generators for the free subgroup),
			\item $\{r_*(ux_{i_1i_2}x_{i_3i_4}\cdots x_{i_{2k-1}i_{2k}})|1\leq i_1<i_2<\cdots<i_{2k}\leq n,\ k\geq 1\}$ ($2^{n-1}-1$ 2-torsion generators), and
			\item $\{r_*(v_{i_1}x_{i_2i_3}x_{i_4i_5}\cdots x_{i_{2k}i_{2k+1}})|1\leq i_1<\cdots<i_{2k+1}\leq n,\ k\geq 1\}$ ($2^{n-1}-n$ 2-torsion generators).
	\end{itemize}
	The reduced $K$-theory $\widetilde{K}^0\left(\coprod_{a\in A}\mathbb{RP}^2\right)$ is isomorphic to  $\bigoplus_{a\in A}\widetilde{K}^0(\mathbb{RP}^2)\oplus \mathbb{Z}^{2^n-1}\cong \mathbb{Z}_2^{2^n}\oplus \mathbb{Z}^{2^n-1}$, where the last summand is generated freely by the trivial line bundle over each of the $2^n-1$ copies of $\mathbb{RP}^2$ which do not contain the basepoint. The map $i^*$ is the reduced restriction $\widetilde{i}^*$ to $\bigoplus_{a\in A}\widetilde{K}^0(\mathbb{RP}^2)$ as in Proposition \ref{restriction2} and the zero map to the summand $\mathbb{Z}^{2^n-1}$. By Proposition \ref{restriction2} again, $\text{im}(i^*)$ is a $1+n+\binom{n}{2}$-dimensional $\mathbb{Z}_2$-vector subspace of $\bigoplus_{a\in A}\widetilde{K}^0(\mathbb{RP}^2)$. By exactness, $\text{im}(d)\cong\left(\bigoplus_{a\in A}\widetilde{K}^0(\mathbb{RP}^2)\oplus\mathbb{Z}^{2^n-1}\right)/\text{im}(i^*)$, which is isomorphic to $\mathbb{Z}_2^{2^n-1-n-\binom{n}{2}}\oplus \mathbb{Z}^{2^n-1}$. We denote a basis of $\mathbb{Z}_2^{2^n-1-n-\binom{n}{2}}$by $f_1, \cdots, f_{2^n-1-n-\binom{n}{2}}$. The short exact sequence 
	\[0\longrightarrow \text{im}(d){\longrightarrow} K^{-1}\left(G/T\times_\Gamma T^n, G/T\times_\Gamma A\right)\stackrel{j^*}{\longrightarrow} K^{-1}(G/T\times_\Gamma T^n)\longrightarrow 0\]
	extracted from the above exact sequence and by Theorem \ref{premainthm} splits, as $j^*\circ r_*$ is the identity map on $K^{-1}(G/T\times_\Gamma T^n)$. It follows that $K^{-1}(G/T\times_\Gamma T^n, G/T\times_\Gamma \{a\})\cong\mathbb{Z}^{2^{n-1}}\oplus\mathbb{Z}_2^{2^n-1-n-\binom{n}{2}}\oplus \mathbb{Z}^{2^n-1}$ with the following generators
	\begin{itemize}
		\item $f_1, \cdots, f_{2^n-1-n-\binom{n}{2}}$, coming from the 2-torsion part of $\text{im}(d)$, 
		\item $\{r_*(w_i)| 1\leq i\leq n\}$, $\{r_*(w_{i_1}x_{i_2i_3}\cdots x_{i_{2k}i_{2k+1}})|1\leq i_1<i_2<\cdots<i_{2k+1}\leq n, k\geq 1\}$ ($2^{n-1}$ free generators), and 
		\item those generating the last summand $\mathbb{Z}^{2^n-1}$, which is the image under $d$ of the trivial vector bundles over each of the $2^n-1$ copies of $\mathbb{RP}^2$ which do not contain the basepoint.
	\end{itemize}
	Note that the second summand of the RHS of the isomorphism (\ref{eqn5}) in Remark \ref{reducedrelative} corresponds to the summand $\mathbb{Z}^{2^n-1}$ of $K^{-1}\left(G/T\times_\Gamma T^n, G/T\times_\Gamma A\right)$. Thus $\widetilde{K}^*(Y_n)$ is generated by the above bulleted list of generators, save the generators in the last bullet point. 
	
	It remains to show that any product of the above generators involving $f_i$ is 0. Note that $j^*: K^0(G/T\times_\Gamma T^n, G/T\times_\Gamma A)\to \widetilde{K}^0(G/T\times_\Gamma T^n)$ is injective because by exactness its kernel is the image of $d$ on $\bigoplus_{a\in A}K^{-1}(\mathbb{RP}^2)$, which is 0. Note that $f_if_j$ and $f_ir_*(w_j)$ are in $K^0(G/T\times_\Gamma T^n, G/T\times_\Gamma A)\cong K^0(Y_n, A)$ and their images under $j^*$ are 0 because $f_i\in \text{im}(d)$ and $j^*(f_i)=0$. It follows that $f_if_j=f_ir_*(w_j)=0$.	 To show that $f_ir_*(\alpha)=0$ where $r_*(\alpha)\in \widetilde{K}^0(Y_n)$, we shall consider the Atiyah-Hirzebruch spectral sequences for $(Y_n, A)$, $G/T\times_\Gamma T^n$ and $G/T\times_\Gamma A$. Note that the spectral sequence for $G/T\times_\Gamma A=\coprod_{a\in A}\mathbb{RP}^2$ collapses on the $E_2$-page. So does the spectral sequence for $G/T\times_\Gamma T^n$ as the Chern character is an isomorphism by Proposition \ref{ringiso}. In particular, if $\text{ch}(\beta)\in H^n(G/T\times_\Gamma T^n, K^0(\text{pt}))=E_2^{n, 0}$, then it survives to $E_\infty^{n, 0}$, which does not have any extension problem, and $\text{ch}(\beta)$ corresponds to $\beta$. From the long exact sequence of the spectral sequences for the pair $(G/T\times_\Gamma T^n, G/T\times_\Gamma A)$
	\[\cdots\longrightarrow E_r^{p, q}(Y_n, A)\stackrel{j^*}{\longrightarrow} E_r^{p, q}(G/T\times_\Gamma T^n)\stackrel{i^*}{\longrightarrow} E_r^{p, q}(G/T\times_\Gamma A)\stackrel{d}{\longrightarrow} E_r^{p+1, q}(Y_n, A)\longrightarrow\cdots,\]
	we can see that the spectral sequence for $(Y_n, A)$ also collapses on the $E_2$-page. Since $f_i\in\text{im}(d: \widetilde{K}^0(G/T\times_\Gamma A)\to K^{-1}(Y_n, A))$, it corresponds to an element 
	\[\widetilde{f}_i\in \text{im}(d: E_2^{2, 0}(G/T\times_\Gamma A)=H^2(G/T\times_\Gamma A, \mathbb{Z})\to E_2^{3, 0}(Y_n, A)=H^3(Y_n, A, \mathbb{Z})).\]
	Similarly, $r_*(\alpha)\in\widetilde{K}^0(Y_n)$ corresponds to an element $\widetilde{\alpha}\in E_2^{p, 0}(Y_n, A)=H^p(Y_n, A, \mathbb{Z})$ for some positive even number $p$. On the one hand, $f_ir_*(\alpha)$ is a 2-torsion of $K^{-1}(Y_n, A)$ and so $f_ir_*(\alpha)\in\text{span}_{\mathbb{Z}_2}\{f_1, f_2, \cdots, f_{2^n-1-n-\binom{n}{2}}\}$, which corresponds to an element in $E_2^{3, 0}(Y_n, A)$. On the other hand, $f_ir_*(\alpha)$ also corresponds to $\widetilde{f}_i\cdot\widetilde{\alpha}\in E_2^{p+3, 0}(Y_n, A)$. Thus $f_ir_*(\alpha)$ can only be 0. This completes the whole proof of the theorem.

\end{proof}
\begin{corollary}\label{cohring}
	The $K$-theory ring $K^*(Y_n)$ is isomorphic to the integral cohomology ring $H^*(Y_n, \mathbb{Z})$ through the Chern character map. In particular, we have 
	\[H^i(Y_n, \mathbb{Z})=\begin{cases}
		\mathbb{Z}&\ \ i=0, \\ 
		0&\ \ i=1, \\ 
		\mathbb{Z}^{\binom{n}{2}}&\ \ i=2, \\ 
		\mathbb{Z}^n\oplus\mathbb{Z}_2^{2^n-1-n-\binom{n}{2}}&\ \ i=3, \\
		\mathbb{Z}^{\binom{n}{i}}\oplus\mathbb{Z}_2^{\binom{n+1}{i-1}}&\ \ i\geq 4\text{ even,}\\ 
		\mathbb{Z}^{\binom{n}{i-2}}&\ \ i\geq 5\text{ odd}.
		\end{cases} \]	
\end{corollary}
\begin{proof}
	Consider the Chern character maps mapping from the rolled-up $K$-theoretic long exact sequence in the proof of Theorem \ref{mainthm} to the following rolled-up cohomological long exact sequence.
	\begin{eqnarray*}
		\resizebox{15cm}{!}{\xymatrix@+4pc{H^\text{even}(G/T\times_\Gamma T^n, G/T\times_\Gamma A)\ar[r]^{j^*}& \widetilde{H}^\text{even}(G/T\times_\Gamma T^n)\ar[r]^{i^*}&\widetilde{H}^\text{even}\left(\coprod_{a\in A}\mathbb{RP}^2\right)\ar[d]^{d}\\
		\bigoplus_{a\in A}H^{\text{odd}}(\mathbb{RP}^2)\ar[u]^d& H^{\text{odd}}(G/T\times_\Gamma T^n)\ar[l]^{i^*}& H^{\text{odd}}(G/T\times_\Gamma T^n, G/T\times_\Gamma A)\ar[l]^{j^*}}}
	\end{eqnarray*}
	The Chern character map $K^*(G/T\times_\Gamma T^n)\to H^*(G/T\times_\Gamma T^n, \mathbb{Z})$ is a ring isomorphism by Proposition \ref{ringiso}, and so is the Chern character map $\widetilde{K}^*\left(\coprod_{a\in A}\mathbb{RP}^2\right)\to \widetilde{H}^*\left(\coprod_{a\in A}\mathbb{RP}^2, \mathbb{Z}\right)$. The Chern character map
	\[\text{ch}: K^*(G/T\times_\Gamma T^n, G/T\times_\Gamma A)\to H^*(G/T\times_\Gamma T^n, G/T\times_\Gamma A)\]
	is a group isomorphism by the 5-lemma and can be described using the commutativity of the following diagram.
	\begin{eqnarray*}
	{\xymatrix{\widetilde{K}^{*-1}(G/T\times_\Gamma A)\ar[r]^d\ar[d]^{\text{ch}}&K^*(G/T\times_\Gamma T^n, G/T\times_\Gamma A)\ar[r]^{j^*}\ar[d]^{\text{ch}}&\widetilde{K}^*(G/T\times_\Gamma T^n)\ar[d]^{\text{ch}}\\ \widetilde{H}^{*-1}(G/T\times_\Gamma A)\ar[r]^d&H^*(G/T\times_\Gamma T^n, G/T\times_\Gamma A)\ar[r]^{j^*}&\widetilde{H}^*(G/T\times_\Gamma T^n)}}
	\end{eqnarray*}
	We describe the Chern character more precisely as follows. For generators $r_*(\alpha)$ of $K^0(G/T\times_\Gamma T^n, G/T\times_\Gamma A)\cong K^0(Y_n, A)$, which are products of the generators in the first four bulleted points in the statement of Theorem \ref{mainthm} and where $\alpha\in \text{ker}(i^*: K^0(G/T\times_\Gamma T^n)\to K^0(G/T\times_\Gamma A))$, we have 
	\[\text{ch}(r_*(\alpha))=(j^*)^{-1}(\text{ch}(\alpha)).\] 
	Note that $(j^*)^{-1}$ is well-defined here because $j^*: H^{\text{even}}(G/T\times_\Gamma T^n, G/T\times_\Gamma A)\to \widetilde{H}^{\text{even}}(G/T\times_\Gamma T^n)$ is injective as its kernel is by exactness the image of $\widetilde{H}^\text{odd}(G/T\times_\Gamma A)=0$ under $d$. Similarly we have 
	\[\text{ch}(r_*(w_ix_{j_1j_2}\cdots x_{j_{2k-1}j_{2k}}))=(j^*)^{-1}\text{ch}(w_ix_{j_1j_2}\cdots x_{j_{2k-1}j_{2k}})\] 
	for $k\geq 1$. We also have
	\[\text{ch}(r_*(w_i))=\pi_i^*(j^*)^{-1}\text{ch}(w),\]
	where $(j^*)^{-1}\text{ch}(w)$ is well-defined because for $n=1$, 
	\[j^*: H^3(G/T\times_\Gamma T, G/T\times_\Gamma A)\cong H^3(G, A)\to H^3(G/T\times_\Gamma T, \mathbb{Z})\]
	is an isomorphism (both cohomology groups are generated by the `volume forms' of $G$ and $G/T\times_\Gamma T$). We do not rewrite $\text{ch}(r_*(w_i))$ as $(j^*)^{-1}\text{ch}(w_i)$ because for $n\geq 3$, $(j^*)^{-1}$ is only defined up to $\text{ker}j^*$, which is the 2-torsion subgroup generated by $\text{ch}(f_1), \cdots, \text{ch}(f_{2^n-1-n-\binom{n}{2}})$. Suppose $f_i=du_i$ for $u_i\in \widetilde{K}^0(G/T\times_\Gamma A)$. Then 
	\[\text{ch}(f_i)=d\text{ch}(u_i)\in H^3(Y_n, A).\]
	Now we shall show that ch is a ring homomorphism and hence a ring isomorphism. This can be checked on the generators of $K^*(Y_n)$.
	\begin{enumerate}
		\item\label{fij} By Theorem \ref{mainthm}, $f_if_j=0$ and so $\text{ch}(f_if_j)=0$. The product $\text{ch}(f_i)\text{ch}(f_j)$ is also 0. That is because 
		\[j^*(\text{ch}(f_i)\text{ch}(f_j))=\text{ch}(j^*f_i)\text{ch}(j^*f_j)=\text{ch}(0)\text{ch}(0)=0\] 
		and $j^*$ is injective on $H^6(G/T\times_\Gamma T^n, G/T\times_\Gamma A)$ as $H^5(G/T\times_\Gamma A, \mathbb{Z})=0$ and by exactness $\text{ker }j^*=\text{im }d=0$. So we have $\text{ch}(f_if_j)=\text{ch}(f_i)\text{ch}(f_j)$. By similar reasonings we also have 
		\[\text{ch}(f_ir_*(w_kx_{j_1j_2}\cdots x_{j_{2\ell-1}j_{2\ell}}))=\text{ch}(f_i)\text{ch}(r_*(w_kx_{j_1j_2}\cdots x_{j_{2\ell-1}j_{2\ell}}))=0\]
		for $\ell\geq 0$.
		\item\label{alpha} Again by Theorem \ref{mainthm}, $f_ir_*(\alpha)=0$. While $\text{ch}(f_i)\text{ch}(r_*(\alpha))$ is a 2-torsion as $f_i$ is, it lies in $H^k(Y_n, A)\cong H^k(Y_n, \mathbb{Z})$ for $k$ odd and $k\geq 5$, which is torsion-free as shown below. So we have
		\[\text{ch}(f_ir_*(\alpha))=\text{ch}(f_i)\text{ch}(r_*(\alpha))=0.\]
		\item\label{alphaw} If $z_1$ and $z_2$ are of the form $\alpha$ or $w_ix_{j_1j_2}\cdots x_{j_{2k-1}j_{2k}}$ for $k\geq 1$, then 
		\begin{align*}
			\text{ch}(r_*(z_1)r_*(z_2))&=\text{ch}(r_*(z_1z_2))\ (r_*\text{ is a homomorphism})\\
								&=(j^*)^{-1}(\text{ch}(z_1z_2))\ ((j^*)^{-1}\text{ is well-defined for ch}(z_1z_2))\\
								&=(j^*)^{-1}(\text{ch}(z_1)\text{ch}(z_2))\ (\text{by Proposition \ref{ringiso}})\\ 
								&=(j^*)^{-1}(\text{ch}(z_1))(j^*)^{-1}(\text{ch}(z_2))\ (j^*\text{ is a homomorphism})\\
								&=\text{ch}(r_*(z_1))\text{ch}(r_*(z_2))
		\end{align*}
		\item For product of the form $r_*(w_i)r_*(z)$ where $z=\alpha$ or $w_ix_{j_1j_2}\cdots x_{j_{2k-1}j_{2k}}$ for $k\geq 0$, its Chern character is $(j^*)^{-1}(\text{ch}(w_i)\text{ch}(z))$ by the same arguments as in item \label{alphaw} above, and it lies in the product of preimages $(j^*)^{-1}(\text{ch}(w_i))(j^*)^{-1}(\text{ch}(z))$. While $(j^*)^{-1}(\text{ch}(z))$ is well-defined, $(j^*)^{-1}(\text{ch}(w_i))$ is equal to $\text{ch}(r_*(w_i))$ up to the subgroup generated by $\text{ch}(f_1)$, $\text{ch}(f_2)$, $\cdots$, $\text{ch}(f_{2^n-1-n-\binom{n}{2}})$. However, $(j^*)^{-1}(\text{ch}(w_i))(j^*)^{-1}(\text{ch}(z))$ in fact is well-defined because 
		\[\text{ch}(f_i)(j^*)^{-1}(\text{ch}(z))=\text{ch}(f_i)\text{ch}(r_*(z))=0\] 
		by items \ref{fij} and \ref{alpha} above, and it is equal to $\text{ch}(r_*(w_i))\text{ch}(r_*(z))$. It follows that 
		\[\text{ch}(r_*(w_i)r_*(z))=\text{ch}(r_*(w_i))\text{ch}(r_*(z)).\]
	\end{enumerate}
	The second claim about the cohomology groups follows from applying the Chern character map to the additive generators of $K^*(Y_n)$ exhibited in Theorem \ref{mainthm}. To be more precise, 
	\begin{itemize}
		\item $H^2(Y_n, \mathbb{Z})$ is generated by the $\binom{n}{2}$ free generators $\{\text{ch}(r_*(x_{ij}))|1\leq i<j\leq n\}$, 
		\item $H^3(Y_n, \mathbb{Z})$ is generated by the $n$ free generators $\{\text{ch}(r_*(w_i))|1\leq i\leq n\}$ and the $2^n-1-n-\binom{n}{2}$ 2-torsion generators $\{\text{ch}(f_i)|1\leq i\leq 2^n-1-n-\binom{n}{2}\}$, 
		\item for $2k\geq 4$, $H^{2k}(Y_n, \mathbb{Z})$ is generated by the $\binom{n}{2k}$ free generators 
		\[\{\text{ch}(r_*(x_{i_1i_2}\cdots x_{i_{2k-1}i_{2k}}))|1\leq i_1<i_2<\cdots<i_{2k}\leq n\},\] 
		the $\binom{n}{2k-2}$ 2-torsion generators 
		\[\{\text{ch}(r_*(ux_{i_1i_2}\cdots x_{i_{2k-4}i_{2k-2}}))| 1\leq i_1<i_2<\cdots<i_{2k-2}\leq n\},\] 
		and the $\binom{n}{2k-1}$ 2-torsion generators 
		\[\{\text{ch}(r_*(v_{i_1}x_{i_2i_3}\cdots x_{i_{2k-2}i_{2k-1}}))|1\leq i_1<i_2<\cdots<i_{2k-1}\leq n\},\] and
		\item for $2k+1\geq 5$, $H^{2k+1}(Y_n, \mathbb{Z})$ is generated by the $\binom{n}{2k-1}$ free generators 
		\[\{\text{ch}(r_*(w_{i_1}x_{i_2i_3}\cdots x_{i_{2k-2}i_{2k-1}}))|1\leq i_1<i_2<\cdots<i_{2k-1}\leq n\}.\]
	\end{itemize}
\end{proof}
\begin{remark}
	It can be shown that the cohomology groups in Corollary \ref{cohring} agree with those deduced from the homotopic decomposition of the suspension $\Sigma Y_n$ given by \cite[Equation (18)]{BJS}. In particular, it is immediate that when $n=3$, our result agrees with that explicitly computed at the end of \cite{BJS}.
\end{remark}
\begin{remark}
	When $n=2$, Theorem \ref{mainthm} agrees with the ordinary $K$-theory ring structure deduced from the $G$-equivariant $K$-theory ring structure given in \cite{Ba}. According to \cite[Equation 5.28]{Ba}, there is the following isomorphism of $R(G)$-algebras
	\[K_G^*(Y_n)\cong R(G)[x_1, x_2, x_3, x_4]/(vx_1-2x_2, x_ix_j\text{ for all }i\text{ and }j), \]
	where $v$ is the standard representation of $G$. Applying the augmentation map to the RHS, we have
	\[\mathbb{Z}[\overline{x}_1, \overline{x}_2, \overline{x}_3, \overline{x}_4]/(2(\overline{x}_1-\overline{x}_2), \overline{x}_i\overline{x}_j\text{ for all }i\text{ and }j),\]
	which indeed is isomorphic to $K^*(Y_2)$ given in Theorem \ref{mainthm}: here the isomorphism sends $\overline{x}_1$ to $r_*(2x_{12})$, $\overline{x}_2$ to $r_*((2+u)x_{12})$, $\overline{x}_3$ to $r_*(w_1)$, and $\overline{x}_4$ to $r_*(w_2)$.\footnote{We believe there are typos on the RHS of the isomorphism in \cite[Equation 5.28]{Ba}. As the LHS is the $G$-equivariant $K$-theory, the RHS should be a $R(G)$-algebra instead of a $R(T)$-algebra. Besides there should be four algebra generators instead of five, because it is mentioned immediately before \cite[Equation 5.27]{Ba} that the author denotes `a generator for each of the four summand in the module structure by $x_i$'.}\end{remark}

\section{Cohomology of $\text{Hom}(\mathbb{Z}^n, SU(2))$ as an FI-module}
	The moduli space $Y_n$ can be endowed with the natural $S_n$-action which permutes the $n$ commuting tuples, making its cohomology group (with coefficient field $\mathbb{F}$) an $S_n$-representation. Moreover, the various cohomology groups for different $n$ are connected by the maps
	\[\varphi_n^*: H^*(Y_n, \mathbb{F})\to H^*(Y_{n+1}, \mathbb{F})\]
	induced by the natural projection onto the first $n$ tuples, which are compatible with the representations of permutation groups in a suitable sense. All these make $H^*(Y_n, \mathbb{F})$ an FI-module over $S_n$ (for definition see \cite{CEF}). It is natural to wonder if the cohomology group is (uniformly) representation stable, i.e., if the decomposition of $H^*(Y_n, \mathbb{F})$ into irreducible representations of $S_n$ stabilizes as $n$ goes to infinity (see \cite{CF} for definition of representation stability). The cohomology $H^i(Y_n, \mathbb{C})$ is uniformly representation stable because, as an $S_n$-representation, it is isomorphic to $\bigwedge\nolimits^i V_{\text{std}}\oplus \bigwedge\nolimits^{i-1}V_{\text{std}}$ when $i$ is even, and $\bigwedge^{i-2}V_{\text{std}}\oplus \bigwedge^{i-3}V_{\text{std}}$ when $i$ is odd (cf. \cite[Section 5.1]{B}). Here $V_{\text{std}}$ stands for the $(n-1)$-dimensional standard representation of $S_n$. It follows that $H^i(Y_n, \mathbb{C})$ is a finitely generated FI-module (cf. \cite[Theorem 1.14]{CEF}). When the coefficient field is $\mathbb{Z}_2$, the FI-module structure behaves very differently.
	\begin{corollary}\label{nonfgfimod}
	The cohomology group $H^*(Y_n, \mathbb{Z}_2)$ is not a finitely generated FI-module.
\end{corollary}
\begin{proof}
	By Corollary \ref{cohring} and the Universal Coefficient Theorem, $\text{dim}_{\mathbb{Z}_2}H^3(Y_n, \mathbb{Z}_2)$ grows exponentially instead of being a polynomial of $n$ eventually. By \cite[Theorem B]{CEFN}, if an FI-module $\{V_n\}$ over any field $\mathbb{F}$ is finitely generated, then $\text{dim }_{\mathbb{F}}V_n$ is a polynomial of $n$ for sufficiently large $n$. Hence $H^3(Y_n, \mathbb{Z}_2)$ is not a finitely generated FI-module.
\end{proof}



\noindent\footnotesize{\textsc{Xi'an Jiaotong-Liverpool University,\\
111 Ren’ai Road, Suzhou Industrial
Park,\\Suzhou, Jiangsu Province 215123, China}\\
\\
\textsc{E-mail}: \texttt{ChiKwong.Fok@xjtlu.edu.cn}\\
\textsc{URL}: \texttt{https://sites.google.com/site/alexckfok}


\begin{thebibliography}{AGPP}
	\bibitem[At]{At} 
	M. F. Atiyah, 
	\emph{K-theory}, 
	Benjamin, New York, 1967.
	
	\bibitem[AB]{AB} 
	M. F. Atiyah, R. Bott, 
	\emph{The Yang-Mills equations over Riemann surfaces}, 
	Phil. Trans. R. Soc. Lond. A 308, 523-615, 1982.
		
	\bibitem[AC]{AC} 
	A. Ad\'em, F. Cohen, 
	\emph{Commuting elements and spaces of homomorphisms}, 
	Mathematische Annalen, Vol. 347, Issue 1, pp 245-248, May 2010. 
	
	\bibitem[ACh]{ACh} 
	A. Ad\'em, M. C. Cheng, 
	\emph{On the moduli spaces of commuting elements in the projective unitary groups}, 
	J. Math. Phys. 60, 071703, 2019.
	
	\bibitem[AG]{AG} 
	A. Ad\'em, J. M. Gom\'ez, 
	\emph{Equivariant K-theory of compact Lie group actions with maximal rank isotropy}, 
	J. Topology 5, 431-57, London Mathematical Society, 2012.
	
	\bibitem[AG2]{AG2} 
	A. Ad\'em, J. M. Gom\'ez, 
	\emph{A classifying space for commutativity in Lie groups}, 
	Algebraic and Geometric Topology, Vol. 15, pp. 493--535, 2015.
	
	\bibitem[AGG]{AGG} 
	A. Ad\'em, J. M. Gom\'ez, S. Gritschacher, 
	\emph{On the second homotopy group of spaces of commuting elements in Lie groups}, 
	Int. Math. Res. Not., rnab 259.
	
	\bibitem[AGPP]{AGPP} A. Ad\'em, J. Ge, J. Pan, N. Petrosyan, 
	\emph{Compatible actions and cohomology of crystallographic groups}, 
	J. Algebra, Vol. 320, no. 1, pp. 341-353, 2008.
	
	\bibitem[B]{B} 
	T. Baird, 
	\emph{Cohomology of the space of commuting n-tuples in a compact Lie group}, 
	Algebraic and Geometric Topology 7, 737-54, 2007.
		
	\bibitem[BJS]{BJS} 
	T. Baird, L. Jeffrey, P. Selick, 
	\emph{The space of commuting n-tuples in $SU(2)$}, 
	Illinois J. Math. Vol. 55, No. 3, 805-813, 2011.
	
	\bibitem[Ba]{Ba}
	T. Bazett, 
	\emph{The equivariant K-theory of commuting 2-tuples in $SU(2)$}, 
	PhD thesis, the University of Toronto, 2016.
	
	\bibitem[Be]{Be} 
	A. Beauville, 
	\emph{Conformal blocks, fusion rules and the Verlinde formula}, 
	Hirzebruch 65 Conference on Algebraic Geometry [Israel Math. Conf. Proc. 9], M. Teicher, ed. (BarIlan University, Ramat Gan 1996), p. 75.
	
	\bibitem[CF]{CF} 
	T. Church, B. Farb, 
	\emph{Representation theory and homological stability}, 
	Advances in Math., Vol. 245, pp. 250-314, 2013.
	
	\bibitem[CEF]{CEF} 
	T. Church, J. Ellenberg, B. Farb, 
	\emph{FI-modules: a new approach to stability for $S_n$-representations}, 
	Duke Math. J., Vol. 164, No. 9, pp. 1833-1910, 2015.
	
	\bibitem[CEFN]{CEFN} 
	T. Church, J. S. Ellenberg, B. Farb, R. Nagpal, 
	\emph{FI-modules over Noetherian rings}, 
	Geometry \& Topology, Vol. 18, pp. 2951-2984, 2014.
	
	\bibitem[CKMS]{CKMS} 
	J.-H. Cho, S. S. Kim, M. Masuda, D. Y. Suh, 
	\emph{Classification of equivariant complex vector bundles over a circle}, 
	J. Math. Kyoto Univ., 41-3, 517-534, 2001.
	
	\bibitem[HL]{HL} 
	M. Harada, G. Landweber, 
	\emph{Surjectivity for Hamiltonian $G$-spaces in $K$-theory}, 
	Trans. Amer. Math. Soc., Vol. 359, no. 12, pp. 6001-6025, 2007.
	
	
	\bibitem[RK]{RK} 
	I. Rosu, with an appendix by I.Rosu and A. Knutson, 
	\emph{Equivariant K-theory and equivariant cohomology}, 
	Math. Z., Vol. 243, 423-448, 2008.
	
	\bibitem[Se]{Se} 
G. Segal, 
\emph{Equivariant $K$-theory}, 
Inst. Hautes \'Etudes Sci. Publ. Math., No. 34, 129-151, 1968.	
	
	
\end{thebibliography}
\end{document}